\documentclass[12pt]{amsart}
\title{NOTAS SOBRE AN\'{A}LISIS FUNCIONAL}
\author{Jaime Chica}
\date{7 de Agosto de 2007}
\address{}
\email{}
\usepackage{amsmath,amssymb,amsthm,pb-diagram,fourier,mathpazo}
\usepackage[activeacute,spanish]{babel}
\newtheorem{defin}{Definici\'{o}n}
\newtheorem{teor}{Teorema}
\newtheorem{corol}{Corolario}
\newtheorem{obser}{Observaci\'{o}n}
\newtheorem{prop}{Proposici\'{o}n}
\newtheorem{ejer}{Ejercicio}
\newtheorem{ejem}{Ejemplo}
\newtheorem{lem}{Lema}
\newtheorem{sol}{Soluciòn}
\newtheorem{con}{Consecuencia}
\begin{document}\maketitle
\begin{abstract}
Estas son las notas de clase no publicadas, sobre el curso de \'{A}nalisis Funcional, impartido por el Prof. Jaime Chica, en la Facultad de Matem\'{a}ticas de la U de A.
\end{abstract}
\tableofcontents
\fontfamily{ppl}\selectfont
Sean $E,F\in\text{Norm}.$\\\\
Llamaremos $\mathcal{L}(E,F)=\bigl\{T:E\longrightarrow F\diagup \text{T es Apli. Lineal}\bigr\}\\\\
\mathcal{L}_c(E,F)=\bigl\{T:E\longrightarrow F\diagup \text{T es A.L Continua}\bigr\}$\\\\
La proposici\'{o}n que sigue es important\'{i}sima en todo el escrito.\\\\
\begin{prop}[$\textbf{\text{Continuidad de una A.L. entre Espacios Normados}}$]
Sean E,F esp. Normados y $T:E\longrightarrow F$ una A.L, i.e, $T\in\mathcal{L}(E,F).$\\\\
Las siguientes afirmaciones son equivalentes:\\
\begin{enumerate}
\item T es continua, o sea, $T\in\mathcal{L}_c(E,F)$\\\\
\item T es continua en 0.\\\\
\item $\exists M>0$ tal que $\forall x\in E:\|T(x)\|\leqslant M\|x\|.$\\\\
\item T es Uniformemente continua.\\\\
\end{enumerate}
\end{prop}
\begin{proof}
$(1)\Rightarrow(2):$ trivial.\\\\
$(2)\Rightarrow(3).$ Supongamos que T es una A.L. continua en 0. Veamos que: $\exists M>0$ tal que \begin{gather}\forall x\in E:\|T(x)\|\leqslant M\|x\|\end{gather}\\\\
Tomemos $\epsilon=1/2$\\\\
Como T es continua en 0 y $\epsilon=1/2>0,\exists\delta>0\,\,\text{tal que $\forall x\in E:\|x\|<\delta\Rightarrow\|T(x)\|<1/2\hspace{0.5cm}\star$}$\\\\
Tomemos $\underset{fijo}{\underbrace{x}}\in E, x\neq0.$\\\\
Entonces $\begin{Vmatrix}\frac{\delta}{2}\frac{x}{\|x\|}\end{Vmatrix}=\frac{\delta}{2}\frac{\|x\|}{\|x\|}=\frac{\delta}{2}<\delta$\\\\
y por $\hspace{0.5cm}\star,\begin{Vmatrix}T\Biggl(\frac{\delta}{2}\frac{x}{\|x\|}\Biggr)\end{Vmatrix}<1/2$\\\\
O sea $\frac{\delta}{2}\frac{1}{\|x\|}\|T(x)\|<1/2,$ i.e: $\|T(x)\|<\frac{1}{\delta}\|x\|.$\\\\
Luego si llamamos $M=\frac{1}{\delta}$ hemos demostrado que $\exists M>0\,\,\text{tal que $\forall x\in E,x\neq0:\|T(x)\|\leqslant M\|x\|$}$\\\\
desigualdad que es obvia para el caso $x=0.$\\\\
$(3)\Rightarrow(4).$ Supongamos ahora que $T:E\longrightarrow F$ es A.L. y que $\exists M>0$ tal que $\forall x\in E:\|T(x)\|\leqslant M\|x\|\hspace{0.5cm}\star\star$\\\\
Veamos que, T es Unif. continua.\\\\
Sea $\epsilon>0.$ Tomemos $\delta=\frac{\epsilon}{M}.$\\\\
Entonces, $\forall x,y\in E$ con $\|x-y\|<\delta=\frac{\epsilon}{M},\|T(x)-T(y)\|=\|T(x-y)\|\underset{\overset{\uparrow}{\star\star}}{\leqslant} M\|x-y\|<M\frac{\epsilon}{M}=\epsilon$\\\\
Esto demuestra que T es Unif. continua.\\\\
$(4)\Rightarrow(1)$ Supongamos ahora que $T:E\longrightarrow F$ es unif. continua. Veamos que T es continua.\\\\
Sea $\epsilon>0.$ Como T es unif. continua, $\exists\delta>0$ tal que $\forall x,y\in E$ con \begin{gather}\|x-y\|<\delta:\|T(x)-T(y)\|<\epsilon\end{gather}\\\\
Tomemos $\underset{fijo}{\underbrace{x_0}}\in E$ y veamos que T es continua en $x_0.$\\\\
Como $\epsilon>0,$ debemos demostrar que $\exists\delta'>0$ tal que $\forall x\in E,\|x-x_0\|<\delta', \|T(x)-T(x_0)\|<\epsilon.$\\\\
Tomemos $\delta'=\frac{\delta}{2}>0.$\\\\
Entonces, $\forall x\in E$ con $\|x-x_0\|<\dfrac{\delta}{2}<\delta,$ se tiene por [2] que $\|T(x)-T(y)\|<\epsilon.$\\\\
As\'{i} que dado $\epsilon>0,\exists\delta>0$ tal que $\forall x\in E,$ si $\|x-x_0\|<\delta$ entonces $\|T(x)-T(x_0)\|<\epsilon,$ lo que demuestra que T es continua en $x_0$ y como $x_0$ es cualquier punto de E, T es continua en E.\\\\
\end{proof}
Sabemos que $\mathcal{L}(E,F)$ es un $\mathbb{K}$ esp. vectorial.\\\\
Podemos ahora probar que $\mathcal{L}_c(E,F)$ es un subespacio de $\mathcal{L}(E,F).$\\
\begin{prop}
Sean E,F esp. normados. Entonces $\mathcal{L}_c(E,F)\subset\mathcal{L}(E,F).$\\
\end{prop}
\begin{proof}
\begin{enumerate}
\item[i)] Consideremos el caso 0.\\
$\begin{diagram}
\node{0:E}\arrow{e,t}\\
\node{F}
\end{diagram}$\\
$\begin{diagram}
\node{x}\arrow{e,t}\\
\node{0_x=0_F}
\end{diagram}$\\\\
Veamos que $0\in\mathcal{L}_c(E,F).$\\\\
Tomemos $\underset{fijo}{\underbrace{M}}>0.$ Entonces,$\forall x\in E:\|0_x\|=\|0_F\|=0\leqslant M\|x\|.$\\\\
Esto demuestra que $0\in\mathcal{L}_c(E,F)$ y por tanto, $\mathcal{L}_c(E,F)\neq\emptyset.$\\\\
\item[ii)] Sean $T_1,T_2\in\mathcal{L}_c(E,F).$ Veamos que $(T_1+T_2)\subset\mathcal{L}_c(E,F)$\\\\
Es claro que $T_1+T_2\in\mathcal{L}(E,F)$\\\\
Resta demostrar que $\exists M>0$ tal que \begin{gather}\label{sum}\forall x\in E:\|(T_1+T_2)_x\|\leqslant M\|x\|\end{gather}\\\\
Como $T_1\in\mathcal{L}_c(E,F),\exists M_1>0$ tal que \begin{gather}\forall x\in E:\|(T_1)_x\|\leqslant M_1\|x\|\end{gather}\\\\
Como $T_2\in\mathcal{L}_c(E,F),\exists M_2>0$ tal que \begin{gather}\forall x\in E:\|(T_2)_x\|\leqslant M_2\|x\|\end{gather}\\\\
Asì que:\\
$\|(T_1+T_2)_x\|=\|T_1(x)+T_2(x)\|\underset{\overset{\nearrow}{(4),(5)}}{\leqslant}\|T_1(x)\|+\|T_2(x)\|\leqslant M_1\|x\|+M_2\|x\|=(M_1+M_2)\|x\|$\\\\
y se tiene ~\eqref{sum}\\\\
\item[iii)] Se tiene $\alpha\in\mathbb{K}$ y $T\in\mathcal{L}_c(E,F).$ Veamos que \begin{gather}\label{eq}(\alpha T)\in\mathcal{L}_c(E,F)\end{gather}\\\\
Si $\alpha=0, \alpha T=0\in\mathcal{L}_c(E,F)$\\\\
Supongamos $\alpha\neq0.$\\\\
Como \begin{gather}T\in\mathcal{L}_c(E,F),\exists M>0\,\,\text{tal que $\forall x\in E:\|T(x)\|\leqslant M\|x\|$}\end{gather}\\\\
Ahora, $\forall x\in E:\|\bigl(\alpha T\bigr)(x)\|=|\alpha|\|T(x)\|\leqslant|\alpha|M\|x\|,$\\\\
lo que nos demuestra ~\eqref{eq}\\\\
\end{enumerate}
\end{proof}

\begin{prop}
Sean E,F:ELN. Si $dimE=n,$ toda AL. $T:E\longrightarrow F$ es continua.\\
\end{prop}
\begin{proof}
Sea \\$\begin{diagram}\node[2]{T:E}\arrow{e,t}\\
\node{F}\end{diagram}$\\
$\begin{diagram}\node[2]{x}\arrow{e,t}\\
\node{T(x)}\end{diagram}$\\\\
T; AL, donde E,F:ELN.\\\\
La funciòn \\$\begin{diagram}\node{{\|\|}^{*}:E}\arrow{e,t}\\
\node{\mathbb{R}}\end{diagram}$\\
$\begin{diagram}\node{x}\arrow{e,t}\\
\node{{\|x\|}^{*}=màx\bigl\{\|x\|_E\|T(x)\|_F\bigr\}}\end{diagram}$\\\\
es una norma en E.\\\\
En efecto, \begin{itemize}\item ${\|x\|}^{*}\geqslant0$\\\\
Si ${\|x\|}^{*}=0,\text{màx$\bigl\{\|x\|_E\|T(x)\|_F\bigr\}=0$}\Rightarrow\|x\|_E=0,\|T(x)\|_F=0$\\\\
\item ${\|\alpha x\|}^{*}=màx\bigl\{\|x\|_E\|T(x)\|_F\bigr\}\\\\
=|\alpha|màx\bigl\{\|x\|_E\|T(x)\|_F\bigr\}\\\\
=|\alpha|{\|x\|}^{*}$\\\\
\item Veamos la desigualdad triangular:\\\\
${\|x+y\|}^{*}=màx\bigl\{\|x+y\|_E\|T(x+y)\|_F\bigr\}\\\\
=\|x+y\|_E\leqslant\|x\|_E+\|y\|_E\\\\
\leqslant màx\bigl\{\|x\|_E\|T(x)\|_F\bigr\}+màx\bigl\{\|y\|_E\|T(y)\|_F\bigr\}\\\\
={\|x\|}^{*}+{\|y\|}^{*}\\\\
\|T(x+y)\|_F=\|T(x)+T(y)\|_F\leqslant\|T(x)\|_F+\|T(y)\|_F\\\\
\leqslant màx\bigl\{\|x\|_E\|T(x)\|_F\bigr\}+màx\bigl\{\|y\|_E\|T(y)\|_F\bigr\}\\\\
={\|x\|}^{*}+{\|y\|}^{*}$\\\\
Como ${\|\|}^{*}$ es una norma en E y $dimE=n,{\|\|}^{*}\sim\|\|_E$ y por tanto,\\ $\exists\beta>0\,\,\text{tal que $\forall x\in E:{\|x\|}^{*}\leqslant\beta\|x\|_E.$}$\\\\
Asì que $\exists\beta>0\,\,\text{tal que $\forall x\in E:$}\\\\
\|T(x)\|_F\leqslant màx\bigl\{\|x\|_E\|T(x)\|_F\bigr\}={\|x\|}^{*}\leqslant\beta\|x\|_E\\\\
{\|x\|}^{*}=màx\bigl\{\|x\|_E\|T(x)\|_F\bigr\}.$\\\\
Lo que nos demuestra que T es continua.\\\\
\end{itemize}
\end{proof}
Esto demuestra que M es cota superior del conjunto de n\'{u}meros reales\\ $\bigl\{\|T(x)\|,\|x\|\leqslant 1\bigr\}.$ Luego $\exists\,\,\underset{\|x\|\leqslant1}{\sup\bigl\{\|T(x)\|\bigr\}}.$\\\\
\begin{defin}
Al n\'{u}mero real $\underset{\|x\|\leqslant1}{\sup\bigl\{\|T(x)\|\bigr\}}$ se le llama la $\underline{\text{Norma de T}}$ y se indica $\|T\|=\underset{\|x\|\leqslant1}{\sup\bigl\{\|T(x)\|\bigr\}}$\\
\end{defin}
\begin{prop}
Sean E y F: ELN. Sabemos que $\mathcal{L}_c(E,F):$ K esp. vect.\\\\
La funciòn $\begin{diagram}\node{\|\|:\mathcal{L}_c(E,F)}\arrow{e,t}\\
\node{\mathbb{R}}\end{diagram}$\\
$\begin{diagram}\node[2]{T}\arrow{e,t}\\
\node{\|T\|=\underset{\|x\|\leqslant1}{\sup\bigl\{\|T(x)\|\bigr\}}}\end{diagram}$\\\\
es una norma en $\mathcal{L}_c(E,F).$\\
\end{prop}
\begin{proof}
\begin{enumerate}
\item Es claro que $\|T\|\geqslant0$\\\\
\item Supongamos que $\|T\|=0.$ Veamos que $T=0$. Entonces $\underset{\|x\|\leqslant1}{\sup\bigl\{\|T(x)\|\bigr\}}=0.$\\\\
O sea que $\forall x\in E$ con $\|x\|\leqslant1:\|T(x)\|=0\Rightarrow T(x)=0.$\\\\
Esto demuestra que \begin{gather}\label{a}\forall x\in E \,\,con \|x\|\leqslant1:\|T(x)\|=0\Rightarrow T(x)=0\end{gather}\\\\
Tomemos ahora $y\in E$\,\, con $y\neq0.$\\\\
$\begin{Vmatrix}\dfrac{y}{\|y\|}\end{Vmatrix}=\dfrac{1}{\|y\|}\|y\|=1.$ Luego por ~\eqref{a}, $\begin{Vmatrix}T\Biggl(\dfrac{y}{\|y\|}\Biggr)\end{Vmatrix}=0$\\\\
O sea que $\dfrac{1}{\|y\|}\begin{Vmatrix}T(y)\end{Vmatrix}=0\Rightarrow\|T(y)\|=0$ i.e, $T(y)=0.$\\\\
Luego $\forall y\in E$ con $y\neq0, T(y)=0.$\\\\
Hemos demostrado as\'{i} que si $\|T\|=0,T$ es la aplicaci\'{o}n cero.\\\\
\item Sea $\alpha\in\mathbb{K}$ y $T\in\mathcal{L}_c(E,F).$ Entonces $(\alpha T)\in\mathcal{L}_c(E,F).$\\\\
Ahora, $\|\alpha T\|=\underset{\|x\|\leqslant1}{\sup\bigl\{\|T(x)\|\bigr\}}=\underset{\|x\|\leqslant1}{\sup\bigl\{\|\alpha T(x)\|\bigr\}}\\\\
=|\alpha|\underset{\|x\|\leqslant1}{\sup\bigl\{\|T(x)\|\bigr\}}=|\alpha|\|T\|.$\\\\
\item Sean $T_1,T_2\in\mathcal{L}_c(E,F).$\\\\
Entonces $T_1+T_2\in\mathcal{L}_c(E,F)$ y por tanto, \begin{gather}\label{b}\|T_1+T_2\|=\underset{\|x\|\leqslant1}{\sup\bigl\{\|\bigl(T_1+T_2\bigr)(x)\|\bigr\}}
=\underset{\|x\|\leqslant1}{\sup\bigl\{\|T_1(x)+T_2(x)\|\bigr\}}\end{gather}\\\\
Ahora, $\|T_1\|=\underset{\|x\|\leqslant1}{\sup\bigl\{\|T_1(x)\|\bigr\}}.$ Luego $\forall x\in E\,\,\text{con $\|x\|\leqslant1:\|T_1(x)\|\leqslant\|T_1\|$}$\\\\
Como $\|T_2\|=\underset{\|x\|\leqslant1}{\sup\bigl\{\|T_2(x)\|\bigr\}}.$ Luego $\forall x\in E\,\,\text{con $\|x\|\leqslant1:\|T_2(x)\|\leqslant\|T_2\|$}$\\\\
Asì que $\forall x\in E:\|T_1(x)+T_2(x)\|\leqslant\|T_1(x)\|+\|T_2(x)\|\leqslant\|T_1\|+\|T_2\|$\\\\
desigualdad que nos muestra que $\|T_1(x)\|+\|T_2(x)\|$ es cota superior del conjunto de lo n\'{u}meros reales $\bigl\{\|T_1(x)+T_2(x)\|,\|x\|\leqslant1\bigr\}$\\\\
Luego $\underset{\|x\|\leqslant1}{\sup\bigl\{\|T_1(x)+T_2(x)\|\bigr\}}\leqslant\|T_1(x)\|+\|T_2(x)\|$ y regresando a ~\eqref{b} se tiene que $\|T_1\|+\|T_2\|\leqslant\|T_1\|+\|T_2\|.$\\
\end{enumerate}
\end{proof}
\begin{prop}
Sean $E,F\in\text{Norm y $T\in\mathcal{L}_c(E,F).$}$\\\\
Entonces, $\forall x\in E:\|T(x)\|\leqslant\|T\|\|x\|.$\\
\end{prop}
\begin{proof}
Sabemos que $\|T\|_{\mathcal{L}_c(E,F)}=\underset{\|x\|\leqslant1}{\sup\bigl\{\|T(x)\|\bigr\}}$\\\\
Luego, \begin{gather}\label{c}\forall x\in E\,\,\text{con $\|x\|_E\leqslant1:\|T(x)\|\leqslant\|T\|$}\end{gather}\\\\
Ahora, $\forall x\in E\,\,\text{con $x\neq0, \begin{Vmatrix}\dfrac{x}{\|x\|}\end{Vmatrix}=1$}$ y al tener en cuenta ~\eqref{c}, $\begin{Vmatrix}T\Biggl(\dfrac{x}{\|x\|}\Biggr)\end{Vmatrix}\leqslant\|T\|$ \\o sea que $\|T(x)\|\leqslant\|T\|\|x\|.$\\
\end{proof}
\begin{ejer}
Sea $E\in\text{Norm}.$\\\\
Demostrar que si una S. de Cauchy en $E=\bigl\{x_n\bigr\}_{n=1}^{\infty}$ tiene una subsucesi\'{o}n convergente a x, la sucesi\'{o}n $x_n\longrightarrow x.$\\
\end{ejer}
\begin{sol}
Sea $\bigl\{x_n\bigr\}_{n=1}^{\infty}\subset E\in\text{Norm}.$ y supongamos que $\bigl\{x_{n_1},x_{n_2},\ldots,x_{n_j},\ldots\bigr\}\subset\bigl\{x_n\bigr\}_{n=1}^{\infty}$ tal que $\lim\limits_{n_j\rightarrow\infty}x_{n_j}=x.$ Veamos que $x_n\longrightarrow x.$\\\\
Sea $\epsilon>0.$ Entonces $\dfrac{\epsilon}{2}>0$ y con $\bigl\{x_n\bigr\}_{n=1}^{\infty}$ es una S. Cauchy en E, \begin{gather}\exists N_1\in\mathbb{N}\text{tal que $\forall m,n>N_1:\|x_n-x_m\|<\dfrac{\epsilon}{2}$}\end{gather}\\\\
Como $\bigl\{x_{n_1},x_{n_2},\ldots,x_{n_j},\ldots\bigr\}\longrightarrow x$ y $\dfrac{\epsilon}{2}>0,$\begin{gather}\label{d}\exists N_2\in\mathbb{N}\text{tal que $\forall n_j>N_2:\|x_{n_j}-x\|<\dfrac{\epsilon}{2}$}\end{gather}\\\\
Tomemos $n>m\'{a}x\bigl\{N_1,N_2\bigr\}.$\\\\
Entonces $\exists ñ_j>N_2\text{tal que $n<ñ_j$}$ y se tiene por ~\eqref{d} que $\|x_n-x_{ñ_j}\|<\dfrac{\epsilon}{2}$\\\\
Como $n,ñ_j>N_1,\|x_n-x_{ñ_j}\|<\dfrac{\epsilon}{2}$\\\\
Luego $\|x_n-x\|=\|\bigl(x_n-x_{ñ_j}\bigr)+\bigl(x_{ñ_j}-x\bigr)\|
\leqslant\|x_n-x_{ñ_j}\|+\|x_{ñ_j}-x\|<\dfrac{\epsilon}{2}+\dfrac{\epsilon}{2}=\epsilon$\\\\
Esto prueba que $\forall n>m\'{a}x\bigl\{N_1,N_2\bigr\}:\|x_n-x\|<\epsilon,$ i.e, $x_n\longrightarrow x.$\\
\end{sol}
\begin{ejer}
Sean $E,F\in\text{Norm}$ y $T\in\mathcal{L}_c(E,F).$ Hemos definido $\|T\|=\underset{\|x\|\leqslant1}{\sup\bigl\{\|T(x)\|\bigr\}}.$\\\\
Probar que $\|T\|=\underset{\|x\|=1}{\sup\bigl\{\|T(x)\|\bigr\}}.$\\
\end{ejer}
\begin{sol}
Es claro que $\bigl\{\|T(x)\|\diagup\|x\|=1\bigr\}\subset\bigl\{\|T(x)\|\diagup\|x\|\leqslant1\bigr\}.$\\\\
Como $\|T\|=\underset{\|x\|\leqslant1}{\sup\bigl\{\|T(x)\|\bigr\}}, \|T\|$ es cota superior del $2^{do}$ cjto.\\\\
Luego $\|T\|$ es cota superior del $1^{er}$ cjto y $\exists\,\, \underset{\|x\|=1}{\sup\bigl\{\|T(x)\|\bigr\}}$ teni\'{e}ndose que\\
$\underset{\|x\|=1}{\sup\bigl\{\|T(x)\|\bigr\}}\leqslant\underset{\|x\|\leqslant1}{\sup\bigl\{\|T(x)\|\bigr\}}=\underset{\|u\|\leqslant1}{\sup\Biggl(\|u\|\begin{Vmatrix}T\Biggl(\dfrac{u}{\|u\|}\Biggr)\end{Vmatrix}\Biggr)}\underset{\overset{\nearrow}{\star}}{\leqslant}\underset{\|x\|=1}{\sup\bigl\{\|T(x)\|\bigr\}}$\\\\
Justifiquemos $\star.$\\\\
Llamemos $S=\bigl\{\|u\|\begin{Vmatrix}T\Biggl(\dfrac{u}{\|u\|}\Biggr)\end{Vmatrix},\|u\|\leqslant1\bigr\}$\\\\
y $W=\bigl\{\|T(y)\|,\|y\|=1\bigr\}.$ Veamos que $\forall s\in S, \exists w\in W$ tal que $s<w$ con lo que se tendr\'{i}a que $\sup S\leqslant\sup W,$ i.e, se tendr\'{i}a $\star.$\\\\
Sea $s\in S.$ Entonces \begin{gather}\label{d}s=\|u\|\begin{Vmatrix}T\Biggl(\dfrac{u}{\|u\|}\Biggr)\end{Vmatrix}\underset{\overset{\nearrow}{\|u\|\leqslant1}}{\leqslant}\begin{Vmatrix}T\Biggl(\dfrac{u}{\|u\|}\Biggr)\end{Vmatrix}\end{gather}\\\\
Tomemos $w=\begin{Vmatrix}T\Biggl(\dfrac{u}{\|u\|}\Biggr)\end{Vmatrix}.$ Es claro que $w\in W $ y que $s\underset{\overset{\nearrow}{~\eqref{d}}}{\leqslant}\begin{Vmatrix}T\Biggl(\dfrac{u}{\|u\|}\Biggr)\end{Vmatrix}=w$\\
\end{sol}
\begin{obser}
Con ligeras variantes podemos demostrar que $\|T\|=\underset{\|x\|<1}{\sup\bigl\{\|T(x)\|\bigr\}}$\\
\end{obser}
\section{Isometr\'{i}as entre E.L.N}
\begin{defin}
Sean $E,F\in Norm.$ y $T:E\longrightarrow F, T\in\mathcal{L}_c(E,F).$\\
\begin{enumerate}
\item T preserva la norma si $\forall x\in E:\|T(x)\|=\|x\|$\\
\item T preserva la distancia si $\forall x,y\in E:d\bigl(T(x),T(y)\bigr)=\|T(x)-T(y)\|=\|x-y\|=d\bigl(x,y\bigr).$\\
\end{enumerate}
\end{defin}
\begin{prop}
\begin{enumerate}
\item Si T preserva la norma, T preseva la distancia.\\
\item Si T preserva la distancia, T preserva la norma.\\
\end{enumerate}
\end{prop}
\begin{proof}
\begin{enumerate}
\item $\forall x,y\in E: d\bigl(T(x),T(y)\bigr)=\|T(x)-T(y)\|=\|T(x-y)\|=\|x-y\|=d\bigl(x,y\bigr).$\\
\item Sea $x\in E.$ Veamos que $\|T(x)\|=\|x\|.\\\\
\|T(x)\|=\|T(x)-0_F\|=\|T(x)-T(0_E)\|=d\bigl(T(x),T(0_E)\bigl)\\\\
=d\bigl(x,0_E\bigr)=\|x-0_E\|=\|x\|.$\\
\end{enumerate}
\end{proof}
\begin{defin}
Sean $E,F\in Norm.$ Una funci\'{o}n $T:E\longrightarrow F$ tal que\\
\begin{enumerate}
\item $T\in\mathcal{L}(E,F).$\\
\item T preserva la norma (\'{o} la distancia) se llama una $\underline{\text{isometr\'{i}a}}$ entre E,F.\\
\end{enumerate}
\end{defin}
\begin{prop}
\begin{enumerate}
\item Toda isometr\'{i}a es 1-1.\\
\item La composici\'{o}n de isometr\'{i}as es una isometr\'{i}a.\\
\item Toda isometr\'{i}a es Unif. continua y por tanto, toda isometr\'{i}a es una funci\'{o}n continua.\\
\item Si T es una isometr\'{i}a entre E y F y T es sobre, $T^{-1}$ es tambi\'{e}n una isometr\'{i}a de F en E.\\
\end{enumerate}
\end{prop}
\begin{proof}
\begin{enumerate}
\item Sean $E,F\in Norm$ y $T:E\longrightarrow F,$ una A.L. tal que $\forall x\in E: \|T(x)\|=\|x\|$ i.e, T es una isometr\'{i}a entre E,F.\\\\
Veamos que $T$ es 1-1 ò que $KerT=\bigl\{0_E\bigr\}$\\\\
Sea $x\in KerT\Rightarrow T(x)=0_F\Rightarrow\|x\|=\|T(x)\|=\|0_F\|=0;$ i.e, $\|x\|=0\Rightarrow x=0_E.$\\
\item [2 y 3] son inmediatas.\\
\item[4] Sean $E,F\in Norm$ y $T:E\longrightarrow F,T\in\mathcal{L}(E,F)$ T: isometr\'{i}a (y por tanto 1-1 y sobre.)\\\\
Entonces \\$\begin{diagram}\node[2]{T^{-1}: F}\arrow{e,t}\\
\node{E}\end{diagram}$\\
$\begin{diagram}\node[2]{y}\arrow{e,t}\\
\node{T^{-1}(y)=x}\end{diagram}$\\\\
donde \begin{gather}\label{e}T(x)=y\end{gather} es A.L. i.e, ${T^{-1}}\in\mathcal{L}(F,E).$\\\\
Veamos que $\forall y\in F:\|T^{-1}(y)\|=\|y\|.$\\\\
Sea $y\in F.$ Entonces $\exists! x\in E$ tal que $T^{-1}(y)=x:\|T^{-1}(y)\|=\|x\|=\|T(x)\|\underset{\overset{\nearrow}{~\eqref{e}}}{=}\|y\|$\\
\end{enumerate}
\end{proof}
\begin{obser}
Si T es una isometr\'{i}a entre E,F; $T\in\mathcal{L}_c(E,F).$ Luego $\|T\|=\underset{\|x\|\leqslant1}{\sup\bigl\{\|T(x)\|\bigr\}}=\underset{\|x\|\leqslant1}{\sup\bigl\{\|x\|\bigr\}}=1$ lo que demuestra que $\underline{\text{la norma de toda isometr\'{i}a es 1.}}$\\
\end{obser}
\begin{defin}[$\textbf{\text{Isomorfismo isom\'{e}trico entre Espacios Normados}}$]
Sean $E,F\in Norm$ y $T:E\longrightarrow F,$ una A.L. biyectiva. Si tanto T como $T^{-1}$ son continuas diremos que T es un isomorfismo (topol\'{o}gico) entre E y F. Y si adem\'{a}s $\forall x\in E: \|T(x)\|=\|x\|$ se dir\'{a} que T es un isomorfismo isom\'{e}trico entre E y F.\\
\end{defin}
La siguiente Prop. da un criterio para establecer cuando una A.L. biyectiva entre E.N. es un Iso. Topol\'{o}gico.\\
\begin{prop}
Sea $E\in Norm, dim E=n.$ Entonces $E\cong\mathbb{K}^n.$\\
\end{prop}
\begin{proof}
Sea $\bigl\{e_1,e_2,\ldots,e_n\bigr\}:$ base de E, y consideremos la A.L.\\
$\begin{diagram}\node[1]{T:\mathbb{K}^n}\arrow{e,t}\\
\node{E}\end{diagram}$\\
$\begin{diagram}\node[1]{(x_1,\ldots,x_n)=x}\arrow{e,t}\\
\node{T(x)=T(x_1,\ldots,x_n)=x_1e_1+\ldots+x_ne_n}\end{diagram}$\\\\
Entonces:\\
\begin{enumerate}
\item T es A.L. Biyectiva y $T^{-1}$ tambi\'{e}n es A.L.\\
\item Veamos que T es continua. Bastar\'{a} con demostrar que $\exists M>0$ tal que $\forall x\in\mathbb{K}^n:\|T(x)\|\leqslant M\|x\|.$\\\\
Tomemos $x=(x_1,\ldots,x_n)\in\mathbb{K}^n$ y consideremos en $\mathbb{K}^n$ la norma $\|x\|=|x_1|+\ldots+|x_n|.$\\\\
Como \begin{equation}\begin{split}\label{f}|x_1|\leqslant\|x_1\|\\
\vdots\\
\underline{|x_n|\leqslant\|x_n\|}\\
\sum\limits_{i=1}^n|x_i|\leqslant n\|x\|\end{split}\end{equation}\\\\
Sea $K=màx\bigl\{\|e_1\|,\ldots,\|e_n\|\bigr\}.$ Entonces $\|T(x)\|=\|T(x_1,\ldots,x_n)\|\\\\=\|x_1e_1+\ldots+x_ne_n\|\\\\
\leqslant|x_1|\|e_1\|+\ldots+|x_n|\|e_n\|\\\\
\leqslant\bigl(|x_1|+\ldots+|x_n|\bigl)K\\\\
=\bigl(\sum\limits_{i=1}^n|x_i|\bigr)K\underset{\overset{\nearrow}{~\eqref{f}}}{\leqslant} nK\|x\|$\\\\
As\'{i} que $\exists M=nK>0$ tal que: $K=m\'{a}x\bigl\{\|e_1\|,\ldots,\|e_n\|\bigr\}.$\\\\
$\forall x\in\mathbb{K}^n=\|T(x)\|\leqslant(nK)\|x\|$ lo que demuestra que $T\in\mathcal{L}_c(\mathbb{K}^n,E).$\\\\
\item Veamos ahora que $T^{-1}:E\longrightarrow\mathbb{K}^n$ es continua.\\\\
Bastar\'{a} con demostrar que $\exists m>0$ tal que $\forall x\in\mathbb{K}^n:m\|x\|\leqslant\|T(x)\|.$\\\\
Consideremos la esfera unidad en $\mathbb{K}^n$ de centro 0 y radio 1:\\\\
$\underset{\text{cerrado y acotado }}{\underbrace{S(0,1)}}=\bigl\{x\in\mathbb{K}^n\diagup\|x\|=1\bigr\}\subset\mathbb{K}^n.$\\\\
Luego por el Teorema de Heine-Borel- Lebesgue, $S(0,1)$ es compacto.\\\\
Consideremos a su vez la funci\'{o}n compuesta definida en el diagrama siguiente:\\\\\\
$
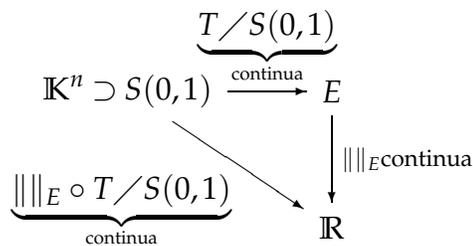
\begin{diagram}
\node[2]{\mathbb{K}^n\supset S(0,1)}\arrow{e,t}{\underset{\text{continua}}{\underbrace{T\diagup S(0,1)}}}\arrow{se,b}{\underset{\text{continua}}{\underbrace{\|\|_E\circ T\diagup S(0,1)}}}\node{E}\arrow{s,r}{\|\|_E\text{continua}}\\
\node[3]{\mathbb{R}}
\end{diagram}$\\\\\\
$\|\|_E\circ T\diagup S(0,1):\\\begin{diagram}\node[1]{S(0,1)\subset\mathbb{K}^n}\arrow{e,t}\\
\node{\mathbb{R}}\end{diagram}$\\
$\begin{diagram}\node[1]{x}\arrow{e,t}\\
\node{\bigl(\|\|_E\circ T\diagup S(0,1)\bigr)(x)=\|T(x)\|_E}\end{diagram}$\\\\
Como $S(0,1)$ es compacto, la funci\'{o}n $\|\|_E\circ T\diagup S(0,1)$ alcanza un valor m\'{i}nimo absoluto en $S(0,1).$\\\\
As\'{i} que $\exists a\in S(0,1)$ tal que \begin{gather}\label{f}\forall x\in S(0,1):\|T(a)\|\leqslant\|T(x)\|\end{gather}\\\\
$a\in S(0,1)\Rightarrow T(a)\neq0\,\,\text{(ya que T es 1-1)$\Rightarrow \|T(a)\|>0$}$\\\\
Tomemos $x\in\mathbb{K}^n, x\neq0.$\\\\
Entonces $\dfrac{x}{\|x\|}\in S(0,1)$ y por tanto, teni\'{e}ndose en cuenta ~\eqref{f}:$$0<\|T(a)\|\leqslant\|T\bigl(\dfrac{x}{\|x\|}\|\bigr)\|x\|=\dfrac{1}{\|x\|}\|T(x)\|$$\\\\
O sea que $\forall x\in\mathbb{K}^n:\|T(a)\|\|x\|\leqslant\|T(x)\|.$\\\\
Esto demuestra que $\exists m=\|T(a)\|>0$ tal que $\forall x\in\mathbb{K}^n:m\|x\|\leqslant\|T(x)\|$ y por la Prop. anterior, $T^{-1}:E\longrightarrow\mathbb{K}^n$ es continua.\\\\
T es asì un $\mathbb{K}^n\cong E$.\\
\end{enumerate}
\end{proof}
Antes de continuar conviene tener presente ciertas propiedades topol\'{o}gicas de los E.N.\\\\
Sea $E\in Norm.$ Entonces $\bigl(E,d\bigr)$ es un Esp. m\'{e}trico, donde $d(x,y)=\|x-y\|.$ A su vez la distancia $d$ induce una topolog\'{i}a sobre E en la que la familia de vecindades $\mathcal{N}_x$ de un punto $x\in E$ se define as\'{i}:\\
$E\supset N\in\mathcal{N}_x\Leftrightarrow \exists\epsilon>0$ tal que $B(0,\epsilon)\subseteq N.$\\\\
O sea que un conjunto $N\subset E$ es vecindad del punto $x\in E$ si el conjunto contiene una bola de centro en el punto; $B(x,\epsilon)=\bigl\{u\in E\diagup d(u,x)=\|u-x\|<\epsilon\bigr\}.$\\\\
Si consideramos a $E$ dotado de la topologìa inducida por la norma, se tienen los siguientes resultados:\\
\begin{prop}
Sea $E\in Norm$ y $A\subset E.$\\
Entonces:\begin{enumerate}\item $\overline{A}=\mathcal{\text{cl}}(A)=\bigl\{x\in E\diagup\forall n\in\mathbb{N}:B(x,\frac{1}{n})\bigcap A\neq\emptyset\bigr\}$\\
\item ${A}^{\circ}=\mathcal{\text{int}}(A)=\bigl\{x\in E\diagup\exists n\in\mathbb{N}:B\bigl(x,\frac{1}{x}\bigr)\subset A\bigr\}$\\
\end{enumerate}
\end{prop}
De esta manera, $$x\in\overline{A}\Leftrightarrow\forall n\in\mathbb{N}:B(x,\frac{1}{x})\bigcap A\neq0$$\\
$$x\in{A}^{\circ}\Leftrightarrow\exists n\in\mathbb{N}:B(x,\frac{1}{x})\subset A$$\\
\begin{proof}
\begin{enumerate}
\item Sea $x\in\mathcal{\text{cl}}(A)=\overline{A}.$\\
Entonces \begin{gather}\label{g}\forall N\in\mathcal{N}_x:N\cap A\neq\emptyset\end{gather}\\
Recordemos que $E\supset N\in\mathcal{N}_x\Leftrightarrow\exists B(x,\epsilon)\subseteq N.$\\
Asì que la familia $\bigl\{B(0,\epsilon),\epsilon>0\bigr\}\subset\mathcal{N}_x.$\\
Luego al tener en cuenta ~\eqref{g}, se tiene que $$\forall n\in\mathbb{N}:B\bigl(x,\frac{1}{n}\bigr)\bigcap A\neq\emptyset.$$\\\\
Sea ahora $x\in E$ con la siguiente propiedad:\\
$\forall n\in\mathbb{N}:B\bigl(x,\frac{1}{n}\bigr)\bigcap A\neq\emptyset.$\\\\
Veamos que: $x\in\mathcal{\text{cl}}(A)$ o que $\forall N\in\mathcal{N}_x:N\cap A\neq\emptyset.$\\\\
Tomemos $\underset{fijo}{\underbrace{N}}\in\mathcal{N}_x$\\\\
$\exists G:\text{abierto, tal que $x\in G\subset N$}\Rightarrow\exists\epsilon>0$ tal que $B(x,\epsilon)\subset G\subset N.$\\\\
Pero si $\epsilon>0,\exists n\in\mathbb{N}$ tal que $\frac{1}{n}<\epsilon\Rightarrow B(x,\frac{1}{n})\subset B(x,\epsilon)\subset N.$\\\\
Ahora, por hipòtesis, $B(x,\frac{1}{n})\bigcap A\neq0.$ i.e, A encuentra a la $B(x,\frac{1}{n})\subset N.$\\\\
Luego A encuentra a N; i.e, $N\cap A\neq\emptyset,$ cualquiera sea $N\in\mathcal{N}_x.$\\\\
\item Se omite.\\
\end{enumerate}
\end{proof}
\begin{prop}
Sea $E\in Norm,$ $A\subset E$ y $x\in E.$\\
Entonces $x\in\overline{A}\Longleftrightarrow\exists\bigl\{x_n\bigr\}_{n=1}^\infty\subset A$ tal que $\lim\limits_{n\rightarrow\infty}x_n=x.$\\
\end{prop}
\begin{proof}
$"\Longrightarrow"$ Sea $x\in\overline{A}.$\\\\
Entonces, por la Prop. anterior, $\forall n\in\mathbb{N}:B(x,\frac{1}{n})\bigcap A\neq\emptyset.$ O sea que $\forall n\in\mathbb{N}:\exists x_n\in A$ tal que $x_n\in B(x,\frac{1}{n}).$\\\\
Esto a su vez quiere decir que $\exists\bigl\{x_n\bigr\}_{n=1}^\infty\subset A$ tal que\\\\
$x_1\in B(x,1)\\\\
x_2\in B(x,\frac{1}{2})\\\\
x_3\in B(x,\frac{1}{3})\\\\
\ldots\ldots\\$\\\\
Es claro que fijado $ñ\in\mathbb{N},\forall n>ñ:x_n\in B(x,\frac{1}{ñ}).$\\\\
Resta demostrar que $x_n\longrightarrow x.$\\\\
Sea $\epsilon>0.$ Entonces $ñ\in\mathbb{N}$ tal que $\frac{1}{ñ}<\epsilon\Longrightarrow B(x,\frac{1}{ñ})\subset B(x,\epsilon).$\\\\
Pero $\forall n>ñ:x_n\in B(x,\frac{1}{ñ})\subset B(x,\epsilon)$, i.e, $\forall n>ñ:x_n\in B(x,\epsilon)$, \\lo que demuestra que $x_n\longrightarrow x.$\\\\
$"\Longleftarrow"$ Supongamos ahora que $x\in E$ tiene esta propiedad: \begin{gather}\label{h}\exists\bigl\{x_n\bigr\}_{n=1}^\infty\subset A\,\,\text{tal que $\lim\limits_{n\rightarrow\infty}x_n=x$}\end{gather}\\\\
Veamos $x\in\overline{A}.$\\\\
 Por la Prop. anterior bastarà con dm. que $\forall n\in\mathbb{N}:B\bigl(x,\frac{1}{n}\bigr)\bigcap A\neq\emptyset.$\\\\
Tomemos $n\in\mathbb{N}.$ Entonces $\frac{1}{n}>0$ y como $x_n\longrightarrow x, \exists N\in\mathbb{N}$ tal que $\forall p>N:x_p\in B(x,\frac{1}{n}).$\\\\
Pero por ~\eqref{h}, $x_p\in A,\forall p>N.$\\\\
Asì que $\exists N\in\mathbb{N}$ tal que $p>N:x_p\in B(x,\frac{1}{n})\bigcap A.$\\\\
Esto demuestra que $\forall n\in\mathbb{N}:B(x,\frac{1}{n})\bigcap A\neq\emptyset.$\\
\end{proof}
Recordemos la definiciòn de sucesiòn convergente en un E. Normado.\\\\
Sea $E\in Norm, \bigl\{x_n\bigr\}_{n=1}^\infty$ sucesiòn en E y $x\in E.$ $x=\lim\limits_{n\rightarrow\infty}x_n\Longleftrightarrow\forall\epsilon>0\\\exists N\in\mathbb{N}\text{tal que $\forall n>N,\|x_n-x\|<\epsilon$}.$\\\\
La prueba de la unicidad del lìmite es trivial.\\\\
Los siguientes hechos se establecen tambièn de manera trivial.\\\\
Sea $E\in Norm,\forall x,y\in E:\left|\|x\|-\|y\|\right|\leqslant\|x\pm y\|\leqslant\|x\|+\|y\|.$\\\\
Si $x_n\longrightarrow x,\alpha x_n\longrightarrow \alpha x.$\\\\
Si $x_n\longrightarrow x, y_n\longrightarrow y, x_n+y_n\longrightarrow x+y.$\\\\
\begin{prop}
La clausura de todo subespacio de un E.L.N. es un subespacio.\\
\end{prop}
\begin{proof}
Sea $E\in Norm$ y $E\supset S:$ subespacio de E. Veamos que $\overline{S}$ es un subespacio de E.\\\\
\begin{enumerate}
\item Como $E\supset S$ y S: subespacio de E, $S\neq\emptyset.$ Ahora $S\subset\overline{S}.$ Luego $\overline{S}\neq\emptyset.$\\\\
\item Sean $x,y\in\overline{S}.$ Veamos que $(x+y)\in\overline{S}.$\\\\
Batarà con demostrar por la Prop. anterior, que $\exists\bigl\{u_n\bigr\}_{n=1}^\infty\subset S.$ tal que $\lim\limits_{n\rightarrow\infty}u_n=x+y.$\\\\
Como $x\in\overline{S}, \exists\bigl\{x_n\bigr\}_{n=1}^\infty\subset S$ tal que $\lim\limits_{n\rightarrow\infty}x_n=x$\\\\
Como $y\in\overline{S},\exists\bigl\{y_n\bigr\}\subset S$ tal que $\lim\limits_{n\rightarrow\infty}y_n=y.$\\\\
Como S es subespacio de E, $\bigl\{u_n\bigr\}=\bigl\{x_n+y_n\bigr\}_{n=1}^\infty\subset S.$\\\\
Luego $\bigl\{u_n\bigr\}_{n=1}^\infty\subset$ y ademàs, $\lim\limits_{n\rightarrow\infty}u_n=\lim\limits_{n\rightarrow\infty}(x_n+y_n)=\lim\limits_{n\rightarrow\infty}x_n+\lim\limits_{n\rightarrow\infty}y_n=x+y$\\\\
\item Sea $x\in\overline{S}$ y $\alpha\in\mathbb{K}.$ Veamos que $(\alpha x)\in\overline{S}.$\\\\
Como $x\in\overline{S},\exists\bigl\{x_n\bigr\}_{n=1}^\infty\subset S$ tal que $\lim\limits_{n\rightarrow\infty}x_n=x\Longrightarrow\bigl\{\alpha x\bigr\}_{n=1}^\infty\subset$ y ademàs, $\lim\limits_{n\rightarrow\infty}(\alpha x_n)=\alpha\lim\limits_{n\rightarrow\infty}x_n=\alpha x.$\\\\
Esto demuestra que $\overline{S}$ es subespacio de E.\\
\end{enumerate}
\end{proof}
\begin{defin}
\begin{enumerate}
\item Sea $E\in Norm$ y $\exists\bigl\{x_n\bigr\}_{n=1}^\infty\subset E.$\\\\
Decimos que $\bigl\{x_n\bigr\}$ es una S. de Cauchy en E si la sucesiòn $\bigl\{x_n\bigr\}$ tiene la siguiente propiedad: $\forall\epsilon>0,\exists N\in\mathbb{N}$ tal que $\forall m,n>N:\|x_m-x_n\|<\epsilon.$\\
\item Sea $E\in Norm.$\\
Decimos que E es un E. de Banach si toda sucesiòn de Cauchy en E converge.\\
\end{enumerate}
\end{defin}
\begin{prop}
Todo subespacio cerrado de un E. de Banach es tambièn de Banach.\\
\end{prop}
\begin{proof}
Es claro que S: Norm. Veamos que S:Banach.\\\\
Sea $\bigl\{x_n\bigr\}_{n=1}^\infty\subset S.$ Veamos que $x_n\longrightarrow x\in S.$\\\\
Como $\bigl\{x_n\bigr\}_{n=1}^\infty\subset S\subset E, \bigl\{x_n\bigr\}_{n=1}^\infty\subset E.$\\\\
Luego, $x_n\longrightarrow x\in E.$ Resta demostrar que $x\in S.$\\\\
Como S es cerrado, $S=\overline{S},$ bastarà con dm. que $x\in\overline{S},$ lo cual resulta claro ya que si $\bigl\{x_n\bigr\}\subset S$ y $x_n\longrightarrow x, x\in\overline{S}.$\\\\
\end{proof}
Vamos a establecer otras propiedades topològicas de los E.N.\\\\
Sea $E\in Norm, a\in E$ y $\gamma>0.$\\\\
Consideremos la $B(a,\gamma)=\bigl\{x\in E\diagup\|x-a\|<\gamma\bigr\}$\\\\
y $S(a,\gamma)=\bigl\{x\in E\diagup\|x-a\|=\gamma\bigr\}.$\\\\
Vamos a demostrar que si E es infinito, estos conjuntos tienen infinitos puntos.\\
\begin{itemize}
\item Tomemos $x\in E, x\neq0.$\\\\
Entonces $\left(a+\dfrac{\gamma}{2\|x\|}x\right)\in B(a,\gamma).$\\\\
En efecto:$\left\|a+\dfrac{\gamma}{2\|x\|}x-a\right\|=\dfrac{r}{2}\left\|\dfrac{x}{\|x\|}\right\|=\dfrac{\gamma}{2}<\gamma$\\\\
lo que demuestra que $\forall x\in E;\left(a+\dfrac{\gamma}{2\|x\|}x\right)\in B(a,\gamma)$ y por tanto, si E es infinito, la $B(a,\gamma)$ contiene $\infty$ puntos.\\\\
\item Ademàs, $\left\|a+\dfrac{\gamma}{2\|x\|}x-a\right\|=\gamma\left\|\dfrac{x}{\|x\|}\right\|=\gamma$ lo que prueba que $\forall x\in E,\\\left(a+\dfrac{\gamma}{\|x\|}x\right)\in S(a,\gamma)$ y de nuevo, si E es infinito, $S(a,\gamma)$ es un conjunto infinito.\\
\end{itemize}
\begin{prop}
Sea $E\in Norm.$\\
La clausura de una bola abierta es la bola cerrada con el mismo radio y centro. O sea: \begin{gather}\label{i}B^*(a,\gamma)=\overline{B(a,\gamma)}\end{gather}\\
\end{prop}
\begin{proof}
Es claro que $B(a,\gamma)\subset B^*(a,\gamma)\hspace{0.5cm}\overline{B(a,\gamma)}\subset\overline{B^*{a,\gamma}}=B^*(a,\gamma)$\\\\
O sea que $\overline{B(a,\gamma)}\subset B^*(a,\gamma).$\\\\
Para tener ~\eqref{i} resta demostrar que $B^*(a,\gamma)\subset\overline{B(a,\gamma)}.$\\\\
Es claro que $B^*(a,\gamma)=B(a,\gamma)\cup S(a,\gamma).$ Sea $y\in B^*(a,\gamma).$ Veamos que $y\in\overline{B(a,\gamma)}.$\\
\begin{enumerate}
\item Supongamos que $y\in\overline{B(a,\gamma)}.$ Como $B(a,\gamma)\subset\overline{B(a,\gamma)}, y\in\overline{B(a,\gamma)}.$\\\\
\item Supongamos que $y\in S(a,\gamma),$ i.e:\begin{gather}\label{j}\|y-a\|=r\end{gather}\\\\
Debemos probar que $y\in\overline{B(a,\gamma)}\Longrightarrow\forall N\in\mathcal{N}_y:N\cap B(a,\gamma)\neq\emptyset.$\\\\
Sea $N\in\mathcal{N}_y.$ Entonces $\exists G:\text{abierto en E tal que $y\in G\subset N\Longrightarrow\exists\gamma >0$}$ \\ tal que $B(y,\gamma)\subset G\subset N.$\\\\
Si logramos probar que $B(y,\gamma)\cap B(a,\gamma)\neq\emptyset,$ y puesto que $B(y\gamma)\subset N$ se tendrìa que $N\cap B(a,\gamma)\neq\emptyset$ como se quiere.\\\\
Tomemos $\underset{fijo}{\underbrace{\epsilon}}<2\gamma.$\\\\
Sea $z=a+\left(1-\dfrac{\epsilon}{2\gamma}\right)(y-a)$ y veamos que $z\in B(y,\gamma), z\in B(a,\gamma).$\\\\
\begin{itemize}
\item $\|z-y\|=\left\|a+\left(1-\dfrac{\epsilon}{2\gamma}\right)(y-a)-y\right\|\\\\
=\left\|\dfrac{\epsilon}{2\gamma}a-\dfrac{\epsilon}{2\gamma}y\right\|=\dfrac{\epsilon}{2\gamma}\|y-a\|\underset{\overset{\nearrow}{~\eqref{j}}}{=}\dfrac{\epsilon}{2\gamma}\gamma=\dfrac{\epsilon}{2}<\gamma$\\\\
Lo que prueba que $z\in B(y,\gamma).$\\\\
\item $\|z-a\|=\left\|a+\left(1-\dfrac{\epsilon}{2\gamma}\right)(y-a)-y\right\|=\dfrac{(2\gamma-\epsilon)}{2\gamma}\|y-a\|\\\\
\underset{\overset{\nearrow}{~\eqref{j}}}{=}\dfrac{\gamma}{2\gamma}(2\gamma-\epsilon)\\\\
=\dfrac{1}{2}(2\gamma-\epsilon)<\dfrac{2\gamma}{2}=\gamma$\\\\
Lo que demuestra que $z\in B(a,\gamma).$\\
\end{itemize}
\end{enumerate}
\end{proof}
\begin{defin}[$\textbf{\text{Conjuntos Convexos en un E.L.N}}$]
Sea $E\in Norm$ y $A\subset E.$\\\\
Decimos que A es convexo si $\forall x,y\in A$ y $\forall\mathsf{t}\in[0,1]$ el vector $\left((1-\mathsf{t}x+\mathsf{t}y)\right)\in A.$\\
\end{defin}
\begin{prop}
Sea $E\in Norm$ y $\underset{convexo}{\underbrace{A}}\subset E.$ Entonces $\overline{A}$ y $A^{\circ}$ son convexos.\\
\end{prop}
\begin{proof}
\begin{enumerate}
\item Por Hip, A es convexo. Veamos que $\overline{A}$ es convexo.\\\\
Tomemos $x,y\in\overline{A}$ y $\mathsf{t}\in[0,1].$\\\\
Bastarà con demostrar que $\left[(1-\mathsf{t}x+\mathsf{t}y)\right]\in \overline{A},$ o que $\exists\bigl\{w_n\bigr\}_{n=1}^\infty\subset A$ tal que \begin{gather}\label{k}\lim\limits_{n\rightarrow\infty}w_n=(1-\mathbf{t})x+\mathsf{t}y\end{gather}\\\\
Como $x\in\overline{A},\exists\bigl\{x_n\bigr\}_{n=1}^\infty\subset \overline{A}$ tal que $\lim\limits_{n\rightarrow\infty}x_n=x.$\\\\
Como $y\in\overline{A},\exists\bigl\{y_n\bigr\}_{n=1}^\infty\subset \overline{A}$ tal que $\lim\limits_{n\rightarrow\infty}y_n=y$\\\\
Fijemos $x\in\mathbb{N}.$ Entonces $x_n,y_n\in A$ y como A es convexo por Hip, $\left[(1-\mathsf{t}x+\mathsf{t}y)\right]\in A$ cualquiera sea $n\in\mathbb{N}.$\\\\
Consiremos ahora la sucesiòn $\bigl\{w_n\bigr\}_{n=1}^\infty\subset A.$\\\\
Ahora, $\lim\limits_{n\rightarrow\infty}w_n=\lim\limits_{n\rightarrow\infty}\left[(1-\mathsf{t}x+\mathsf{t}y)\right]=(1-\mathsf{t})x+\mathsf{t}y$ y se tiene ~\eqref{k}.\\\\
\item Por Hip. A es convexo. Veamos que $A^\circ$ es convexo.\\\\
De nuevo se toman $x,y\in A^\circ$ y $\mathsf{t}\in[0,1].$ Veamos que $\left[(1-\mathsf{t}x+\mathsf{t}y)\right]\in A^\circ\Longrightarrow\exists\gamma>0$ tal que $B\left((1-\mathsf{t})x+\mathsf{t}y,\gamma\right)\subset A.$\\\\
Como $x\in A^\circ,\exists\gamma_1>0$ tal que $B(x,\gamma_1)\subset A.$\\\\
Como $y\in A^\circ,\exists\gamma_2>0$ tal que $B(y,\gamma_2)\subset A.$\\\\
Tomando $\gamma<mìn\bigl\{\gamma_1,\gamma_2\bigr\}$ se tiene que $B(x,\gamma)\subset A, B(y,\gamma)\subset A.$\\\\
Consideremos la $B\left((1-\mathsf{t})x+\mathsf{t}y,\gamma\right).$ Veamos que $B\left((1-\mathsf{t})x+\mathsf{t}y,\gamma\right)\subset A.$\\\\
Sea $w\in\left((1-\mathsf{t})x+\mathsf{t}y,\gamma\right).$\\\\
Entonces \begin{gather}\label{l}\left\|w-\left[(1-\mathsf{t})x+\mathsf{t}y\right]\right\|<\gamma\end{gather}\\\\
Definamos $x_1=x+w-\left[(1-\mathsf{t})x+\mathsf{t}y\right].$ Entonces $x_1-x=w-\left[(1-\mathsf{t})x+\mathsf{t}y,\gamma\right]$ y por tanto, $\|x_1-x\|=\left\|w-\left[(1-\mathsf{t})x+\mathsf{t}y\right]\right\|<\gamma$, lo que nos demuestra que $x_1\in B(x,\gamma)\subset A,$ i.e, $\boxed{x_1\in A}.$\\\\
Definamos $y_1=y+w-\left[(1-\mathsf{t})x+\mathsf{t}y\right].$ Entonces $y_1-y=w-\left[(1-\mathsf{t})x+\mathsf{t}y,\right]$ y por tanto, $\|y_1-y\|=\left\|w-\left[(1-\mathsf{t})x+\mathsf{t}y\right]\right\|<\gamma$, lo que nos demuestra que $y_1\in B(y,\gamma)\subset A,$ i.e, $\boxed{y_1\in A}.$\\\\
Como $\mathsf{t}\in[0,1]$ y $x_1,y_1\in A;$ convexo se tiene que $(1-\mathsf{t})x_1+\mathsf{t}y_1\in A.$\\\\
O sea que $(1-\mathsf{t})\left\{x+w-\left[(1-\mathsf{t})x+\mathsf{t}y\right]\right\}+\mathsf{t}\left\{y+w-\left[(1-\mathsf{t})x+\mathsf{t}y\right]\right\}\in A $\\\\
i.e, $(1-\mathsf{t})+\mathsf{t}y+w-\left[(1-\mathsf{t})x+\mathsf{t}y\right]\in A.$\\\\
Lo que nos demuestra que $w\in A.$\\
\end{enumerate}
\end{proof}
\begin{prop}
Toda bola abierta ò cerrada de un E.N. es un conjunto convexo.\\
\end{prop}
\begin{proof}
$E\in Norm,a\in E$ y $\gamma>0.$\\\\
Consideremos la $B(a,\gamma)=\bigl\{x\in\diagup\|x-a\|<\gamma\bigr\}.$ Veamos que $B(a,\gamma):\text{convexo}.$\\\\
Tomemos $x,y\in B(a,\gamma)$ y $\underset{fijo}{\underbrace{\mathsf{t}}}\in[0,1].$\\\\
Entonces \begin{gather}\label{l}\|x-a\|<\gamma,\|y-a\|<\gamma\end{gather}\\\\
Entonces $(1-\mathsf{t})x+\mathsf{t}y\in B(a,\gamma)\Longrightarrow\left\|(1-\mathsf{t})x+\mathsf{t}y-a\right\|<\gamma.$\\\\
$\left\|(1-\mathsf{t})x+\mathsf{t}y-a\right\|<\gamma=\left\|(1-\mathsf{t})x+\mathsf{t}a+\mathsf{t}y-a-\mathsf{t}a\right\|\\\\
=\left\|(1-\mathsf{t})(x-a)+\mathsf{t}(y-a)\right\|\\\\
\leqslant(1-\mathsf{t})\|x-a\|+\mathsf{t}\|y-a\|\underset{\overset{\nearrow}{~\eqref{l}}}{<}(1-\mathsf{t})\gamma+\mathsf{t}\gamma=\gamma$\\\\
Ahora, $B^*(a,\gamma)=\overline{B(a,\gamma)}.$ Como $B(a,\gamma)$ es convexo, $\overline{B(a,\gamma)}$ es convexo y por lo tanto, $B^*(a,\gamma)$ es convexo.\\\\
\end{proof}
El siguente lema serà ùtil al $\underline{\text{completar}}$ un E.N. $(E,\|\|).$\\
\begin{lem}[$\textbf{\text{Un criterio para establecer cuando un E.L.N. es de Banach}}$]
Sea $E\in Norm.$ Si $\exists A\subset E$ tal que:\\
\begin{enumerate}
\item $\overline{A}=E$ i.e. A es denso E.\\
\item $\forall\underset{\text{S. de Cauchy}}{\underbrace{\bigl\{x_n\bigr\}}}\subset A,\bigl\{x_n\bigr\}$ converge en E.\\
Entonces $E\in Ban.$\\
\end{enumerate}
\end{lem}
\begin{proof}
Sea $\bigl\{x_n\bigr\}_{n=1}^\infty:$ S. de Cauchy en E. Veamos que $\bigl\{x_n\bigr\}$ converge en E.\\\\
Tomemos un tèrmino $\underset{fijo}{\underbrace{x_N}}$ de la sucesiòn.\\\\
Entonces $x_N\in E=\overline{A}$ (Hip.) Luego $x_N\in\overline{A}$ y por tanto, $\forall n\in\mathbb{N}:B(x_N,\frac{1}{n})\bigcap A\neq\emptyset.$\\\\
Si $n=1,\exists\beta_1^N\in A\,\,\text{tal que $\|\beta_1^N-x_N\|<1$}$\\\\
Si $n=2,\exists\beta_2^N\in A\,\,\text{tal que $\|\beta_2^N-x_N\|<\frac{1}{2}$}$\\\\
Si $n=3,\exists\beta_3^N\in A\,\,\text{tal que $\|\beta_3^N-x_N\|<\frac{1}{3}$}$\\\\
$\text{etc}\ldots\ldots\ldots$\\\\
De este modo, si variamos a N:\\\\
para $x_1,\exists\bigl\{\beta_n^1\bigr\}_{n=1}^\infty\subset A\,\,\text{tal que $\|\beta_1^1-x_1\|<1$}$\\\\
$\|\beta_2^1-x_1\|<\frac{1}{2}$\\\\
$\|\beta_3^1-x_1\|<\frac{1}{3}$\\\\
$\vdots$\\\\
para $x_2,\exists\bigl\{\beta_n^2\bigr\}_{n=1}^\infty\subset A\,\,\text{tal que $\|\beta_1^2-x_2\|<1$}$\\\\
$\|\beta_2^2-x_2\|<\frac{1}{2}$\\\\
$\|\beta_3^2-x_2\|<\frac{1}{3}$\\\\
$\vdots$\\\\
para $x_3,\exists\bigl\{\beta_n^3\bigr\}_{n=1}^\infty\subset A\,\,\text{tal que $\|\beta_1^3-x_3\|<1$}$\\\\
$\|\beta_2^3-x_3\|<\frac{1}{2}$\\\\
$\|\beta_3^3-x_3\|<\frac{1}{3}$\\\\
$\vdots$\\\\
Consideremos la sucesiòn $\bigl\{\beta_n^n\bigr\}_{n\in\mathbb{N}}\subset A.$\\\\
Se tiene que $\forall n\in\mathbb{N}:\|\beta_n^n-x_n\|<\frac{1}{n}.$\\\\
Se deja como ejercicio al lector probar que la sucesiòn $\bigl\{\beta_n^n\bigr\}$ es una S. de Cauchy en A (y por la Hip. convergerà en E).\\
\end{proof}
Continuando entonces con las propiedades de los E.L.N; presentamos ahora, uno de los resultamos màs impotantes del Anàlisis Funcional.\\\\
\begin{prop}\label{i}
En un espacio vectorial de dimensiòn finita, todas las normas son equivalentes.\\
\end{prop}
\begin{proof}

(\textit{La prueba que haremos es tomada del texto, Funtional Analysis, by Bachman-Narici})\\\\
Sea X;K esp. vec., $dimX=n$ y sea $\|\|$ una norma cualquiera en X.\\\\
\begin{enumerate}
\item[\text{Paso 1.}] Vamos a construir una cierta norma $\|\|_0$ en X.\\
\item[\text{Paso 2.}] Enseguida demostraremos que $\|\|\sim\|\|_0.$\\\\
Esto demostrarà que todas las normas en X son equivalentes.\\\\
\item[\text{Paso 1.}] Tomemos $\underset{fija}{\underbrace{\bigl\{X_1,X_2,\ldots,X_n\bigr\}}}:$ Base de X.\\\\
Sea $x\in X.$ Entonces $\exists!\alpha_1,\ldots,\alpha_n\in K$ tal que $x=\alpha_1X_1+\ldots+\alpha_nX_n.$\\\\
Esto permite que podamos definir la funciòn\\
$\begin{diagram}
\node[2]{\|\|_0:X}\arrow{e,t}\\
\node{\mathbb{R}}
\end{diagram}$\\
$\begin{diagram}
\node[2]{x}\arrow{e,t}\\
\node{\|x\|_0=\underset{i=1,\ldots,n}{màx|\alpha_i|}}
\end{diagram}$\\\\
Es fàcil probar que $\|\|_0$ es una norma en X. La llamaremos \\$\underline{\text{norma cero en X asociada a la Base $\bigl\{X_\alpha\bigr\}$}}.$ (Si cambiamos de base, cambia la representaciòn del vector y por lo tanto, cambia la norma.)\\\\
\item[\text{Paso 2.}] Sea $\|\|$ una norma cualquiera en X. Nuestra tarea es demostrar que $\exists a,b>0$ tal que: \begin{gather}\label{n}\forall x\in X:a\|x\|_0\leqslant\|x\|\leqslant b\|x\|_0\end{gather}\\\\
En efecto: $\left\|x\|=\|\alpha_1X_1+\ldots+\alpha_nX_n\right\|\leqslant\|x\|_0\|X_1\|+\ldots\ldots+\|X\|_0\|X_n\|\leqslant\\
\underset{b}{\underbrace{\left(\|X_1\|+\ldots\ldots+\|X_n\|\right)}}\|X\|_0$\\\\\\
$\begin{cases}
\|X\|_0=\underset{i=1,\ldots,n}{màx|\alpha_i|}\\\\
x=\alpha_1X_1+\ldots+\alpha_nX_n\\\\
\therefore\hspace{0.5cm}|\alpha_1|\leqslant\underset{i=1,\ldots,n}{màx|\alpha_i|}=\|X\|_0\\\\
\hspace{1.5cm}\vdots\\\\
|\alpha_n|\leqslant\underset{i=1,\ldots,n}{màx|\alpha_i|}=\|X\|_0
\end{cases}$\\\\
Esto demuestra la desigualdad de la derecha en ~\eqref{n}.\\\\
Resta demostrar que $\exists a>0$ tal que $\forall x\in X: a\|x\|_0\leqslant\|x\|\hspace{0.5cm}\star$\\\\
Esto ya no es tan simple!!\\\\
Procederemos por inducciòn sobre la dimensiòn de X.\\\\
\item Supongamos que $dim X=1.$\\\\
Sea $\bigl\{x_1\bigr\}:$ Base de X. Tomemos $x\in X.$ Entonces $x=\alpha_1X_1.$\\\\
$\therefore\|X\|_0=|\alpha_1|.$ Sea ahora $\|\|$ una norma cualquiera en X.\\\\
$\|X\|=\|\alpha_1X_1\|=|\alpha|_1\|X_1\|=\|X_1\|\|X\|_0$ y por tanto, $\underset{a}{\underbrace{\|X_1\|}}\|X\|_0\leqslant\|X\|$.\\\\
\item Hip. de Inducciòn:\\
Asumimos que la propiedad es cierta para espacios de demensiòn $p-1.$\\\\
Sea X: K esp. vectorial con $dimX=p$ y sea $\|\|$ una norma cualquiera en X. El plan es demostrar que \begin{gather}\label{o}\exists a>0\text{ tal que $\forall x\in X: a\underset{i=1,\ldots,n}{màx|\alpha_i|}=a\|X\|_0\leqslant\|X\|.$}\end{gather}\\\\
Tomemos $M=Sg\bigl\{x_1,\ldots,x_{n-1}\bigr\}.$\\\\
Es claro que siendo $M\subset(X,\|\|), \|\|$ es una norma en M, $\bigl\{x_1,\ldots,x_{n-1}\bigr\}:$ Base de M y $dimM=p-1.$\\\\
Como $dimM=p-1,$ se tiene, por la Hip. de Inducciòn que toda norma en M es $\sim$ a la norma $\|\|_0$ en M asociada a la Base de M:$\bigl\{x_1,\ldots,x_{n-1}\bigr\}.$\\\\
Ahora, $\|\|$ es una norma en M. Por tanto, $\|\|\sim\|\|_0$ y se tiene que \begin{gather}\label{p}\exists a>0\,\,\text{tal que $\forall y\in M:a\|y\|_0\leqslant\|y\|\Longrightarrow\|y\|_0=\underset{i=1,\ldots,n}{màx|\alpha_i|}$}\end{gather}\\\\
Compare ~\eqref{o} y ~\eqref{p}: Observe los cuantificadores.\\\\
Vamos ahora a demostrar que $(M,\|\|):$ Banach.\\\\
Sea $\bigl\{y_n\bigr\}\subset(M,\|\|).$ Veamos que \begin{gather}\label{q}y_n\overset{\|\|}{\longrightarrow} y\in M\end{gather}\\\\
Como $\|\|\sim\|\|_0$ en M, $\bigl\{y_n\bigr\}$ es una S. de Cauchy en $(M,\|\|_0)$\\\\
Denotemos $\bigl\{y_n\bigr\}=\bigl\{y_1,y_1,\ldots,y_n\bigr\}\subset M.$\\\\
Ahora, como $\bigl\{x_1,\ldots,x_{n-1}\bigr\}:$ Base de M, podemos escribir:\\\\
$y_1=\alpha_1^1x_1+\alpha_1^2x_2+\ldots\ldots+\alpha_1^{n-1}x_{n-1}\\\\
y_2=\alpha_2^1x_1+\alpha_2^2x_2+\ldots\ldots+\alpha_2^{n-1}x_{n-1}\\\\
y_3=\alpha_3^1x_1+\alpha_3^2x_2+\ldots\ldots+\alpha_3^{n-1}x_{n-1}\\\\
\centerline{\vdots}$\\\\
Considremos ahora las suces. columna en K $(\mathbb{R}\text{ò $\mathbb{C}$})$:\\\\
\begin{tabular}{ccc}
$\bigl\{\alpha_j^1\bigr\}_{j=1}^\infty$&$\bigl\{\alpha_j^2\bigr\}_{j=1}^\infty\ldots\ldots$&$\bigl\{\alpha_j^{n-1}\bigr\}_{j=1}^\infty$\\\\
$\alpha_1^1$&$\alpha_1^2\ldots\ldots$&$\alpha_1^{n-1}$\\\\
$\alpha_2^1$&$\alpha_2^2\ldots\ldots$&$\alpha_2^{n-1}$\\\\
$\alpha_3^1$&$\alpha_3^2\ldots\ldots$&$\alpha_3^{n-1}$\\\\
$\downarrow$&$\downarrow\ldots\ldots$&$\downarrow$\\\\
\end{tabular}
\end{enumerate}
Vamos a demostrar que c/u de estas $(n-1)$ Sucesiones converge en K.\\\\
Sea $\epsilon>0.$ \\\\Como $\bigl\{y_n\bigr\}:$ S. de Cauchy en $(M,\|\|_0), \exists N\in\mathbb{N}\text{tal que $\forall p,q>N:\underset{i=1,\ldots,n-1}{màx|\alpha_p^i-\alpha_q^i|=\|y_p-y_q\|_0<\epsilon}$}$\\\\
$y_p=\alpha_p^1x_1+\ldots\ldots+\alpha_p^{n-1}x_{n-1}\\\\
y_q=\alpha_q^1x_1+\ldots\ldots+\alpha_q^{n-1}x_{n-1}$\\\\
Fijemos $p,q>N.$\\\\
Entonces:\\\\
$\begin{cases}
|\alpha_p^1-\alpha_q^1|\leqslant \underset{i=1,\ldots,n-1}{màx|\alpha_p^i-\alpha_q^i|}<\epsilon, & \text{siempre que $p,q>N$}\\\\
\vdots\\\\
|\alpha_p^{n-1}-\alpha_q^{n-1}|\leqslant \underset{i=1,\ldots,n-1}{màx|\alpha_p^i-\alpha_q^i|}<\epsilon
\end{cases}$\\\\
Esto significa que\\\\
$\begin{cases}
\text{la sucesiòn $\bigl\{\alpha_j^1\bigr\}_{j=1}^\infty$ es una S de Cauchy en K y como K es completo, $\alpha_j^1\overset{j\rightarrow\infty}{\longrightarrow}\alpha_1$}\\\\
\ldots\ldots\\\\
\text{la sucesiòn $\bigl\{\alpha_j^{n-1}\bigr\}_{j=1}^\infty$ es una S de Cauchy en K y como K es completo, $\alpha_j^{n-1}\overset{j\rightarrow\infty}{\longrightarrow}\alpha_{n-1}$}
\end{cases}$\\\\
Definamos $y=\alpha_1x_1+\ldots\ldots+\alpha_{n-1}x_{n_1}.$\\\\
Es claro que $y\in M.$\\\\
Para tener ~\eqref{q}, veamos que $y_n\overset{\|\|}{\longrightarrow} y.$\\\\
Bastarà con demostrar que $y_n\overset{\|\|_0}{\longrightarrow} y$ ya que siendo $\|\|\sim\|\|_0$ en M, $y_n\overset{\|\|}{\longrightarrow} y.$\\\\
Veamos pues, que $y_n\overset{\|\|_0}{\longrightarrow} y.$\\\\
Sea $\epsilon>0.$\\\\
Debemos demostrar que $\exists N\in\mathbb{N}$ tal que $\forall q>N:\underset{i=1,\ldots,n}{màx|\alpha_q^i-\alpha_i|}=\|y_q-y\|_0<\epsilon$\\\\
$\begin{cases}
\text{Dado que $\alpha_j^1\overset{j\rightarrow\infty}{\longrightarrow}\alpha_1$ y que $\epsilon>0,$ se tiene que}\\\\
\exists N_1\in\mathbb{N}\,\,\text{tal que $\forall q>N_1:|\alpha_q^i-\alpha_i|<\epsilon$}\\\\
\ldots\ldots\\\\
\text{Como $\alpha_j^{n-1}\overset{j\rightarrow\infty}{\longrightarrow}\alpha_{n-1}$ y que $\epsilon>0,$ se tiene que}\\\\
\exists N_{n-1}\in\mathbb{N}\,\,\text{tal que $\forall q>N_{n-1}:|\alpha_q^{n-1}-\alpha_{n-1}|<\epsilon$}
\end{cases}$\\\\
Luego si escogemos $N>màx\bigl\{N_1,\ldots,N_{n-1}\bigr\},$ se tiene que $\forall q>N$:$\begin{cases}|\alpha_q^1-\alpha_1|<\epsilon\\\\
\ldots\ldots\\\\
|\alpha_q^{n-1}-\alpha_{n-1}|<\epsilon.
\end{cases}$\\\\
y por lo tanto, $\underset{i=1,\ldots,n-1}{màx|\alpha_q^i-\alpha_i|}<\epsilon$ siempre que $q>N.$\\\\
Pero $\underset{i=1,\ldots,n-1}{màx|\alpha_q^i-\alpha_i|}=\|y_q-y\|_0.$\\\\
Luego $\|y_q-y\|_0<\epsilon$ siempre que $q>N.$\\\\
Esto demuestra que $y_n\overset{\|\|_0}{\longrightarrow} y$ y por tanto, $y_n\overset{\|\|}{\longrightarrow} y.$\\\\
Hemos demostrado asì que $(M,\|\|):$ Esp. Banach.\\\\
Vamos ahora a probar que $M\subset X$ es un cjto cerrado en X.\\\\
Debemos probar que $\overline{M}=M.$\\\\
Es claro que $M\subset\overline{M}.$ Resta demostrar que $\overline{M}\subset M.$\\\\
Tomemos $x\in\overline{M}$ y veamos que $x\in M.$ $x\in\overline{M}\Longrightarrow\exists\bigl\{x_n\bigr\}\subset M\text{tal que $\lim_{n\rightarrow\infty} x_n=x\in X.$}$\\\\
Pero $(M,\|\|)\in Banach$ y como $\bigl\{ x_n\bigr\}\subset M,$ entonces $\bigl\{x_n\bigr\}$ es una S. de Cauchy en $(M,\|\|)\Longrightarrow\lim_{n\rightarrow\infty} x_n\in M.$\\\\
O sea que $x\in M.$ Esto prueba que M es un cjto cerrado en X.\\\\
Consideremos ahora el conjunto $x_n+M=\bigl\{x_n+z,z\in M\bigr\}\subset X.$ \\$x_n+M,$ no es subespacio de X ya que $0\notin x_n+M.$ En efecto, si $0\in x_n+M, \exists\bigl(\beta_1x_1+\ldots+\beta_{n_1}x_{n-1}\bigr)\in M$ tal que $0_x=x_n+\beta_1x_1+\ldots,\beta_{n-1}x_{n-1}$ y se tendrìa que $x_n$ es una C.L. de $\bigl\{x_1,\ldots,x_{n-1}\bigr\},\rightarrow\leftarrow$ ya que el cjto. $\bigl\{x_1,\ldots,x_n\bigr\}$ es L.I.\\\\
Considremos ahora en X la traslaciòn definida por el vector $-x_n:$\\
$\begin{diagram}
\node[2]{T_{-x_n}:(X,\|\|)}\arrow{e,t}\\
\node{(X,\|\|)}
\end{diagram}$\\
$\begin{diagram}
\node[2]{x}\arrow{e,t}\\
\node{T_{-x_n}(x)=x-x_n}
\end{diagram}$\\\\
Es fàcil demostrar que T es continua. Ahora, como M es cerrado en X, $T_{-x_n}^{-1}(M)$ es cerrado en X.\\\\
Pero $T_{-x_{n}}^{-1}(M)=\bigl\{x\in\text{tal que $T_{-x_{n}}(x)=x-x_n\in M$}\bigr\}\\\\
=\bigl\{x\in\text{tal que $x-x_n=\alpha_1x_1+\ldots+\alpha_{n-1}x_{n-1},\alpha_i\in K $}\bigr\}\\\\
=\bigl\{x\in\text{tal que $\underset{\in M }{\underbrace{x=\alpha_1x_1+\ldots+\alpha_{n-1}x_{n-1},\alpha_i+x_n}}\in K $}\bigr\}\\\\
=x_n+M.$\\\\
Por lo tanto, $\underset{cerrado}{\underbrace{x_n+M}}\subset(X,\|\|)\Longrightarrow$ $C _x(x_n+M):$ abto en X.\\\\
Como $0\in C_x(x_n+M),\exists c_n>0$ tal que $B_{(x,\|\|)}(0;c_n)\subset C_x(x_n+M).$\\\\
i.e, $\exists c_n>0$ tal que $\forall x\in X:\text{si $\|x\|<c_n,$ entonces, $x\in C_x(x_n+M)$};$ i.e $x\notin x_n+M.$\\\\
Tomando contrarecìproco, tendremos que $\forall x\in X;$ si $x\in x_n+M$ entonces $\|x\|\geqslant c_n.$\\\\
Tomemos ahora $\underset{fijo}{\underbrace{x}}\in x_n+M.$ Entonces $x=\alpha_1x_1+\ldots+\alpha_{n-1}x_{n-1}+x_n, \alpha_i\in K.$ Y se tiene que:\\
$\|\alpha_1x_1+\ldots+\alpha_{n-1}x_{n-1}+x_n\|\geqslant c_n,$ cualquiera sean $\alpha_1,\alpha_{n-1}\in K.$\\\\
Hemos demostrado asì que $\exists c_n>0$ tal que $\forall\alpha_1,\alpha_{n-1}\in K:c_n\leqslant\|\alpha_1x_1+\ldots+\alpha_{n-1}x_{n-1}+x_n\|.$ De aquì resulta claro que $\forall\underset{\neq0}{\alpha_n}\in K,c_n\leqslant\left\|\dfrac{\alpha_1}{\alpha_n}x_1+\ldots+\dfrac{\alpha_{n-1}}{\alpha_n}x_{n-1}+x_n\right\|,$\\\\
o lo que es lo mismo, $|\alpha_n|c_n\leqslant\|\alpha_1x_1+\ldots+\alpha_{n-1}x_{n-1}+\alpha_nx_n\|$ lo que nos dice que $\forall x\in X:|\alpha_n|c_n\leqslant\|x\|,$ donde $\alpha$ es la cta $n^{a}$ del vector x en la base $\bigl\{x_1,\ldots,x_n\bigr\}.$\\\\
Resumiendo, hemos demostrado que si $M=Sg(x_1,\ldots,x_{n-1}),$ se tiene que $\exists c_n>0$ tal que $\forall x\in X:|\alpha_n|c_n\leqslant\|x\|$ donde $\alpha$ es la cta $n^{a}$ del vector x en la base $\bigl\{x_1,\ldots,x_n\bigr\}$.\\\\
Si ahora tomamos como $M=Sg(x_1,x_2,\ldots,x_{n-2},x_n)$ y desarrollamos el mismo anàlisis que se hizo para el caso en que M era el $Sg(x_1,\ldots,x_{n-1}),$ podrìamos demostrar que $\exists c_{n-1}>0$ tal que $\forall x\in X:|\alpha_{n_1}|c_{n_1}\|x\|$ donde $\alpha_{n-1}$ es la cta $(n-1)$ del vector x en la base $\bigl\{x_1,\ldots,x_n\bigr\}.$\\\\
etc$\ldots,\ldots$\\\\
Asì que en el fondo de todo lo que se tiene es que $\exists c_1,c_2,\ldots,c_n>0$ tal que $\forall x\in X:$\\\\
$\begin{cases}
|\alpha_1|c_1\leqslant\|x\|;\alpha_1:\text{cte $1^a$ de X en la base $\bigl\{x_1,\ldots,x_n\bigr\}$}\\\\
|\alpha_2|c_2\leqslant\|x\|;\alpha_1:\text{cte $2^a$ de X en la base $\bigl\{x_1,\ldots,x_n\bigr\}$}\\\\
$\vdots$\\\\
|\alpha_n|c_n\leqslant\|x\|;\alpha_n:\text{cte $n^a$ de X en la base $\bigl\{x_1,\ldots,x_n\bigr\}$}
\end{cases}$\\\\
Fijemos x en X. Entonces $x=\beta_1x_1+\ldots+\beta_{n-1}x_{n-1}+\beta_nx_n$ y se tiene que:\\\\
$\begin{cases}
|\beta_1| \underset{j=1,\ldots,n}{mìn C_j}\leqslant\|x\|\\\\
|\beta_2| \underset{j=1,\ldots,n}{mìn C_j}\leqslant\|x\|\\\\
\vdots\\\\
|\beta_n| \underset{j=1,\ldots,n}{mìn C_j}\leqslant\|x\|
\end{cases}$\\\\
Y si llamamos $a=\underset{j=1,\ldots,n}{mìn C_j}$ se tiene que $\forall i=1,\ldots,n: a|\beta_i|\leqslant\|x\|\Longrightarrow a\underset{i=1,\ldots,n}{màx|\beta_i|}\leqslant\|x\|.$ Pero $\underset{i=1,\ldots,n}{màx|\beta_i|}=\|x\|_0$\\\\
Asì que $a\|x\|_0\leqslant\|x\|,$ y como x es cualquier vector en X, hemos conseguido demostrar que $\exists a>0$ tal que $\forall x\in X:a\|x\|_0\leqslant\|x\|$ que era plan que nos habìamos propuesto en ~\pageref{i}.\\
\end{proof}
\begin{corol}
Sea $(X,\|\|):$ E.L.N. Si $E\subset X$ y $dimE=n,$ entonces $(E,\|\|):$ Banach.\\
\end{corol}
\begin{proof}
Es claro que siendo $E\subset X,$ $(X,\|\|):$ E.L.N.\\\\
Sea $\bigl\{y_n\bigr\}_{n=1}^\infty$ una S. de Cauchy en $(E,\|\|).$ Veamos que $y_n\overset{\|\|}{\longrightarrow} y\in E.$\\\\
Como $dim E=n,$ sea $\underset{fija}{\bigl\{x_1,\ldots,x_n\bigr\}}:$ base de E.\\\\
Tomemos $\mathsf{v}\in E.$ Entonces $\exists!\alpha_1,\ldots,\alpha_n\in K$ tal que $\mathsf{v}=\alpha_1x_1,\ldots,\alpha_nx_n$ y podemos definir\\
$\begin{diagram}
\node[2]{\|\|_0:E}\arrow{e,t}\\
\node{\mathbb{R}}
\end{diagram}$\\
$\begin{diagram}
\node[2]{\mathsf{v}}\arrow{e,t}\\
\node{\|\mathsf{v}\|_0=\underset{i=1,\ldots,n}{màx|\alpha_i|}}
\end{diagram}$\\\\
Es fàcil probar que $\|v\|_0$ es una norma en E. Asì que $(E,\|\|_0):$ E.L.N\\\\
Dado que $\|\|$ y $\|\|_0$ son normas en E y como $dimE=n,$ se tiene, por la Prop. anterior que $\|\|\sin\|\|_0.$\\\\
Como $\bigl\{y_n\bigr\}_{n=1}^\infty\subset(E,\|\|),$ entonces, $\bigl\{y_n\bigr\}$ es una S. Cauchy en $(E,\|\|_0).$\\\\
Si logramos demostrar que $y_n\overset{\|\|_0}{\longrightarrow}y\in E,$ entonces, $y_n\overset{\|\|}{\longrightarrow}y\in E$ que es lo que se quiere probar.\\\\
Asì que todo consiste en probar que $y_n\overset{\|\|_0}{\longrightarrow}y\in E.$\\\\
Sea $\varepsilon>0.$ Como $\bigl\{y_n\bigr\}:$ S. de Cauchy en $(E,\|\|_0),$ \begin{gather}\label{r}\exists N\in\mathbb{N}\,\, \text{tal que $\forall p,q>N:\|y_p-y_q\|_0<\varepsilon$}\end{gather}\\\\
Como $\bigl\{x_1,\ldots,x_n\bigr\}:$ Base de E y $\bigl\{y_n\bigr\}_{n=1}^\infty\subset E,$ los tèrminos de la sucesiòn $y_n$ se pueden escribir asì:\\\\
$y_1=\alpha_1^1x_1+\alpha_1^2x_2+\ldots\ldots+\alpha_1^nx_n\\\\
y_2=\alpha_2^1x_1+\alpha_2^2x_2+\ldots\ldots+\alpha_2^nx_n\\\\
\ldots\ldots\\\\
y_N=\alpha_N^1x_1+\alpha_N^2x_2+\ldots\ldots+\alpha_N^nx_n\\\\
\ldots\ldots\\\\
y_p=\alpha_p^1x_1+\alpha_p^2x_2+\ldots\ldots+\alpha_p^nx_n\\\\
y_q=\underset{\downarrow}{\alpha_q^1x_1}+\underset{\downarrow}{\alpha_q^2x_2}+\ldots\ldots+\underset{\downarrow}{\alpha_q^nx_n}$\\\\
Regresemos a ~\eqref{r}.\\\\
Tomemos $\underset{fijos}{\underbrace{p,q}}>N.$ Entonces $\|y_p-y_q\|_0<\varepsilon.$\\\\
$\left\|\alpha_p^1x_1+\alpha_p^2x_2+\ldots\ldots+\alpha_p^nx_n-\bigl(\alpha_q^1x_1+\alpha_q^2x_2+\ldots\ldots+\alpha_q^nx_n\bigr)\right\|_0\\\\
=màx\bigl\{|\alpha_p^1-\alpha_q^1|,\ldots\ldots,|\alpha_p^n-\alpha_q^n|\bigr\}$\\\\
Asì que\\\\
$\begin{cases}
\left|\alpha_p^1-\alpha_q^1\right|<\varepsilon,\text{siempre que $p,q>N$}\\\\
\ldots\ldots\\\\
\left|\alpha_p^n-\alpha_q^n\right|<\varepsilon,\text{siempre que $p,q>N$}
\end{cases}$\\\\\\
$\begin{cases}
\text{la sucesiòn  de columnas $\bigl\{\alpha_j^1\bigr\}_{j=1}^\infty$ es una S. de Cauchy en $K(\mathbb{R}\,\,\text{ò}\mathbb{C})$}\\\\
\text{y como K es completo, $\alpha_1^j\overset{j\rightarrow\infty}{\longrightarrow}\alpha_1$}\\\\
\ldots\ldots\\\\
\text{la sucesiòn  de columnas $\bigl\{\alpha_j^n\bigr\}_{j=1}^\infty$ es una S. de Cauchy en $K(\mathbb{R}\,\,\text{ò}\mathbb{C})$}\\\\
\text{y como K es completo, $\alpha_j^n\overset{j\rightarrow\infty}{\longrightarrow}\alpha_n$}
\end{cases}$\\\\
Definamos ahora el vector $y=\alpha_1x_1+\ldots\ldots+\alpha_nx_n$\\\\
Resta demostrar que $y\in E $ (lo cual es obvio) y que $y_n\overset{\|\|_0}{\longrightarrow} y$\\\\
Sea $\varepsilon>0.$ Debemos demostrar que $\exists N\in\mathbb{N}$ tal que \begin{gather}\label{s}\forall j>N:màx\bigl\{|\alpha_j^1-\alpha_n|,\ldots\ldots,|\alpha_j^n-\alpha_n|\bigr\}=\left\|y_j-y\right\|_0<\varepsilon\end{gather}\\\\
Como $\alpha_j^1\overset{j\rightarrow\infty}{\longrightarrow}\alpha_1,\exists N_1\in\mathbb{N}$ tal que $\forall j>N_1:\left|\alpha_j^1-\alpha_1\right|<\varepsilon$\\\\
Como $\alpha_j^2\overset{j\rightarrow\infty}{\longrightarrow}\alpha_2,\exists N_2\in\mathbb{N}$ tal que $\forall j>N_2:\left|\alpha_j^2-\alpha_2\right|<\varepsilon$\\\\
$\ldots\ldots$\\\\
Como $\alpha_j^n\overset{j\rightarrow\infty}{\longrightarrow}\alpha_n,\exists N_n\in\mathbb{N}$ tal que $\forall j>N_n:\left|\alpha_j^n-\alpha_n\right|<\varepsilon$\\\\
Luego si escogemos $N>màx\bigl\{N_1,\ldots,N_n\bigr\}$ se tendrà que \\\begin{gather*}\forall j>N:\left|\alpha_j^1-\alpha_1\right|<\varepsilon\\\\
\left|\alpha_j^2-\alpha_2\right|<\varepsilon\\\\
\vdots\\\\
\underline{\left|\alpha_j^n-\alpha_n\right|<\varepsilon}\\\\
\therefore\hspace{0.5cm}màx\bigl\{\left|\alpha_j^i-\alpha_1\right|,\ldots\ldots,\left|\alpha_j^n-\alpha_n\right|<\varepsilon, \text{siempre que $j>N$ y se tiene ~\eqref{s}.} \end{gather*}\\\\
\end{proof}
\begin{corol}
Todo E.L.N. de dimensiòn finita es un E. de Banach.\\
\end{corol}
\begin{corol}
Sea $(X,\|\|):$ E.L.N. $E\subset X.$ Entonces:\\
\begin{enumerate}
\item $(E,\|\|):$ Banach.\\\\
\item E es cerrado en X. \\O sea que $\textit{todo subespacio vectorial de dimensiòn finita de un E.L.N. es un cjto cerrado. }$\\
\end{enumerate}
\end{corol}
\begin{proof}
\begin{enumerate}
\item $\checkmark$\\\\
\item Debemos demostrar que $\overline{E}=E.$ $E\subset\overline{E}\checkmark.$ Resta probar que $\overline{E}\subset E.$\\\\
Tomemos $x\in\overline{E}$ y veamos que $x\in E.$ $x\in\overline{E}\Longrightarrow\exists\bigl\{x_n\bigr\}\subset E$ tal que $\lim\limits_{n\rightarrow\infty}x_n=x(x\in X).$\\\\
Pero $(E,\|\|):$Banach y como $\bigl\{x_n\bigr\}\subset(E,\|\|),\bigl\{x_n\bigr\}$ es una S. de Cauchy en $(E,\|\|)\hspace{0.5cm}\therefore\left(\lim\limits_{n\rightarrow\infty}x_n\right)\in E.$ O sea que $x\in E.$\\
\end{enumerate}
\end{proof}
\begin{ejem}
Sabemo que $\mathbb{C}:$Esp. vectorial.\\\\
Vamos a demostrar que $(\mathbb{C},\|\|):$ E.L.N. donde $\forall z\in\mathbb{C}:\|z\|=\sqrt{x^2+y^2}=|z|.$\\\\
Una vez demostremos que $(\mathbb{C,\|\|}):$ E.L.N, como $dim\mathbb{C}=1,$ se tendrà que $(\mathbb{C,\|\|}):$ E.L.N. de dimensiòn finita. Luego $(\mathbb{C,\|\|}):$ Banach.\\\\
\begin{enumerate}
\item Es claro que $|z|\geqslant0.$ Si $z=0, |z|=|0|=0.$\\\\
Supongamos $|z|=0.$ Entonces $\sqrt{x^2+y^2}=0\Rightarrow x=y=0;z=0.$\\\\
\item $|\alpha z|^2=(\alpha z)\overline{(\alpha z)}=\alpha z\overline{\alpha}\overline{z}=(\alpha\overline{\alpha})(z\overline{z})=|\alpha|^2|z|^2.$\\\\
\item $|z_1+z_2|\leqslant|z_1|+|z_2|\\\\
|z_1+z_2|^2=(z_1+z_2)(\overline{z_1+z_2})\\\\
=(z_1+z_2)(\overline{z_1}+\overline{z_2})\\\\
=z_1\overline{z_1}+z_2\overline{z_2}+(z_1\overline{z_2}+\overline{z_1}z_2)\\\\
=|z_1|^2+|z_2|^2+2\mathsf{Re}(z_1\overline{z_2})\leqslant|z_1|^2+|z_2|^2+2|z_1\overline{z_2}|\\\\
=|z_1|^2+|z_2|^2+2|z_1||z_2|\\\\
=\left(|z_1|+|z_2|\right)^2\Longrightarrow|z_1+z_2|\leqslant|z_1|+|z_2|$\\
\end{enumerate}
\end{ejem}
\begin{ejem}
Es posible dm. de una manera directa (i.e, sin utilizar el Tma. "todo E.V. de dim. finita es Banach".) que $(\mathbb{C},\|\|)$ es Banach.\\\\
$|z|=\sqrt{x^2+y^2}$ siendo $z=x+iy$\\
\end{ejem}
\subsection{El teorema de Hahn-Banach. (caso real)}
Este importante Tma. del An. Funcional establece que todo funcional lineal continuo definido sobre un subespacio de un E.L.N. siempre se puede extender a todo el espacio conservàndose la norma del funcional.\\\\
Comenzamos con un\\
\begin{lem}
Sea $(X,\|\|):\mathbb{R}$ esp.vectorial normado, $M\subset X$ y $f\in M'=\mathcal{L}_c(M,\mathbb{R})\diagup\text{dual topològico de M}.$\\\\
Sea $x_0\in X,x_0\notin M.$\\\\
Entonces $\exists g\in S_g(M\cup\bigl\{x_0\bigr\}),$ o sea $g:S_gM\cup\bigl\{x_0\bigr\}\longrightarrow\mathbb{R}$\\\\
tal que:\\
\begin{enumerate}
\item $g\diagup M=f$\\
\item $\|g\|=\|f\|.$\\
\end{enumerate}
\end{lem}
\begin{proof}
Tomemos $\underset{fijo}{\underbrace{y_1}}\in M.$\\\\
Entonces, $\forall y\in M:\\\\
\left||f(y_1)|-|f(y)|\right|\leqslant\left|f(y_1)-f(y)\right|=\left|f(y_1-y)\right|\leqslant\|f\|\|y_1-y\|\\\\
=\|f\|\|(y_1+x_0)+(-y-x_0)\|\leqslant\|f\|\|y_1+x_0\|+\|f\|\|y+x_0\|\\\\
\therefore\\\\
|f(y)|-\|f\|\|y_1+x_0\|-\|f\|\|y+x_0\|\leqslant f(y_1)\leqslant|f(y)|+\|f\|\|y_1+x_0\|+\|f\|\|y+x_0\|\\\\
-|f(y)|-\|f\|\|y+x_0\|\leqslant\|f\|\|y_1+x_0\|-\|f(y_1)\|$ lo que dm. que el real es cota superior del cjto $\left\{-|f(y)|-\|f\|\|y+x_0\|,y\in M\right\}\\\\
\therefore\hspace{0.5cm}\exists Sup\left\{-|f(y)|-\|f\|\|y+x_0\|,y\in M\right\}=a\\\\
-\|f(y_1)\|-\|f\|\|y_1+x_0\|\leqslant|f(y)|+\|f\|\|y+x_0\|$ lo que dm. que el real es cota inferior del cjto $\left\{-|f(y)|-\|f\|\|y+x_0\|,y\in M\right\}.$\\\\
Luego $\exists \text{ìnf}\left\{-|f(y)|-\|f\|\|y+x_0\|,y\in M\right\}=b.$\\\\
Veamos que $a\leqslant b.$\\\\
Recordemos que el Sup de un cjto de nùmeros reales es la mìnima cota superior del cjto.\\\\
Si logramos dm. que b es cota superior del cjto $\left\{-|f(y)|-\|f\|\|y+x_0\|,y\in M\right\},$ tendremos que $a\leqslant b.$\\\\
Asì que vamos a dm. que b es cota superior del cjto citado, o lo que es lo mismo, que \begin{gather}\label{t}\forall y\in M:-|f(y)|-\|f\|\|y+x_0\|\leqslant b\end{gather}\\\\
Razonemos por R. abs.\\\\
O sea, supongamos que ~\eqref{t} no es cierta.\\\\
Entonces $\exists \overset{\sim}{y}\in M$ tal que $b<-|f(\overset{\sim}{y})|-\|f\|\|\overset{\sim}{y}+x_0\|\\\\
\therefore\hspace{0.5cm} 0<\varepsilon=-|f(\overset{\sim}{y})|-\|f\|\|\overset{\sim}{y}+x_0\|-b$\\\\
Pero $b=\text{ìnf}\left\{-|f(y)|-\|f\|\|y+x_0\|\right\}$ y como $\varepsilon>0,$ se tiene, por la propiedad de aproximaciòn del ìnf, que $\exists\overset{\sim}{\overset{\sim}{y}}\in M$ tal que $-|f(\overset{\sim}{\overset{\sim}{y}})|-\|f\|\|\overset{\sim}{\overset{\sim}{y}}+x_0\|<b+\varepsilon\\\\
=-|f(\overset{\sim}{y})|-\|f\|\|\overset{\sim}{y}+x_0\|\\\\
\therefore\\\\
\|f\|\left(\|\overset{\sim}{\overset{\sim}{y}}+x_0\|+\|\overset{\sim}{y}+x_0\|\right)=\|f\|\|\overset{\sim}{y}+x_0\|+\|f\|\|\overset{\sim}{\overset{\sim}{y}}+x_0\|<\left|f(\overset{\sim}{\overset{\sim}{y}})-f(\overset{\sim}{y})\right|\\\\
\|f\|\|\overset{\sim}{\overset{\sim}{y}}-\overset{\sim}{y}\|\geqslant\left|f(\overset{\sim}{\overset{\sim}{y}}-\overset{\sim}{y})\right|=\left|f(\overset{\sim}{\overset{\sim}{y}})-f(\overset{\sim}{y})\right|\geqslant\left||f(\overset{\sim}{\overset{\sim}{y}})|-|f(\overset{\sim}{y})|\right|$\\\\
O sea que $\left||f(\overset{\sim}{\overset{\sim}{y}})|-|f(\overset{\sim}{y})|\right|<|f(\overset{\sim}{\overset{\sim}{y}})|-|f(\overset{\sim}{y})|,(\rightarrow\leftarrow)$\\\\
Luego $a\leqslant b.$\\\\
Supongamos que $a<b$ y tomemos $a<\underset{fijo}{\underbrace{c}}<b.$\\\\
Como $\text{Sup}\left\{-|f(y)|-\|f\|\|y+x_0\|\right\}=a<c<b=\text{ìnf}\left\{-|f(y)|+\|f\|\|y+x_0\|\right\}$\\\\
se tiene que $\forall m\in M:\\\\
-|f(m)|-\|f\|\|m+x_0\|<c<-|f(m)|+\|f\|\|m+x_0\|$\\\\
O sea que $-\|f\|\|m+x_0\|<f(m)+c<\|f\|\|m+x_0\|,$ i.e \begin{gather}\label{u}\left|f(m)+c\right|\leqslant\|f\|\|m+x_0\|\end{gather}\\\\
cualquiera sea $m\in M.$\\\\
Ahora, es fàcil dm. que $S_g\left(M\cup\bigl\{x_0\bigr\}\right)=\left\{m+\lambda x_0, m\in M,\lambda\in\mathbb{R}\right\}.$\\\\
Definamos\\
$\begin{diagram}
\node[2]{g:S_g M\cup\bigl\{x_0\bigr\}}\arrow{e,t}\\
\node{\mathbb{R}}
\end{diagram}$\\
$\begin{diagram}
\node[2]{m+\lambda x_0}\arrow{e,t}\\
\node{g\left(m+\lambda x_0\right)=f(m)+\lambda c}
\end{diagram}$\\\\
Es fàcil dm. que $g$ es A.L. $\text{Hàgalo!}$\\\\
Veamos que $g\diagup M=f.$\\\\
Tomemos $m\in M$ y veamos que $g(m)=f(m).$\\\\
$m=m+0.x_0$ y $g(m)=g(m+0.x_0)=f(m)+0.c=f(m).$\\\\
Veamos que $g$ es continua.\\\\
Tomemos $x\in S_g\left(M\cup\bigl\{x_0\bigr\}\right).$ Entonces $x=m+\lambda x_0.$ Asumamos $\lambda\neq0.$\\\\
$|g(x)|=|g(m+\lambda x_0)|=|f(m)+\lambda c|=|\lambda f({\lambda}^{-1}m)+\lambda c|\\\\
=|\lambda||f({\lambda}^{-1}m)+c|\underset{~\eqref{u}}{\leqslant}|\lambda|\|f\|\|{\lambda}^{-1}m+x_0\|\\\\
=\|f\|\|m+\lambda x_0\|\\\\
=\|f\|\|x\|\hspace{0.5cm}\star.$\\\\
(Si $\lambda=0 \text{se llega a lo mismo}$)\\\\
Asì que $\forall x\in M:|g(x)|\leqslant\|f\|\|x\|$ lo que prueba que $g$ es continua. Resta dm que $\|g\|=\|f\|.$\\\\
Veamos que \begin{gather}\label{v}\|g\|\leqslant\|f\|\end{gather}\\\\
Como $\|g\|=\underset{\|x\|\leqslant1}{\text{Sup}|g(x)|},$ para obtener ~\eqref{v} basta con dm. que $\|f\|$ es cota superior del cjto $S=\left\{g(x), \|x\|\leqslant1\diagup x\in S_g(M\cup\bigl\{x_0\bigr\})\right\}.$\\\\
Sea $\xi\in S.$ Veamos que $\xi\leqslant\|f\|.\xi\in S\Longrightarrow\exists x=m+\lambda x_0\in S_g(M\cup\bigl\{x_0\bigr\})$ con $\|x\|\leqslant1$ tal que $\xi=|g(x)|.$\\\\
Si tenemos en cuenta $\star$ se tiene que:\\\\
$\xi=|g(x)|\leqslant\|f\|\|x\|\leqslant\|f\|\,\,\text{l.q.q.d}$\\\\
Veamos finalmente que \begin{gather}\label{w}\|f\|\leqslant\|g\|\end{gather}\\\\
$\|f\|=\underset{\|x\|\leqslant1}{\text{Sup}|f(x)|}.$ Para obtener ~\eqref{w} basta dm. que $\|g\|$ es cota superior del cjto $T=\left\{|f(x)|,\|x\|\leqslant1,x\in M\right\}.$\\\\
Sea $\xi\in T.$ Veamos que $\xi\leqslant\|g\|.$ $\xi\in T\Longrightarrow\exists x\in M$ con $\|x\|\leqslant1$ tal que $g=|f(x)|.$\\\\
Por tanto, $g(x)=g(x+0.x_0)=f(x)+0.c=f(x).\\\\
\therefore\hspace{0.5cm}|g(x)|=|f(x)|=g,$ i.e, \begin{equation}\label{x} |g(x)|=\xi\end{equation}\\\\
Como $x\in S_g(M\cup\bigl\{x_0\bigr\})$ y $\|g\|=\underset{\|y\|\leqslant1}{\text{Sup}|g(y)|}$\\\\
$|g(x)|\leqslant\|g\|$ y regresando a ~\eqref{x} se tiene que $\xi\leqslant\|g\|.$\\
\end{proof}
\begin{prop}[\textbf{\text{El Teorema de Hahn-Banach}}]
Sea $(X,\|\|):\mathbb{R}$ esp. vectorial normado, $M\subset X$ y sea $f\in M'=\mathcal{L}_c(M,\mathbb{R}).$ O sea, $f:M\longrightarrow\mathbb{R}\diagup\text{A.L. continua.}$\\\\
Entonces $\exists F\in X'=\mathcal{L}_c(X,\mathbb{R}),$ i.e, $F:X\longrightarrow\mathbb{R}$ tal que:\\
\begin{enumerate}
\item $F\diagup M=f$\\
\item $\|F\|=\|f\|$\\\\
\end{enumerate}
Y en palabras, "todo funcional lineal continuo definido sobre un subespacio de un E.L.N se puede extender (prolongar) a todo el espacio preservando la norma."\\
\end{prop}
\begin{proof}
Sea S=$\begin{cases}
(1). \overline{f}:M\subset D_{\overline{f}}\longrightarrow\mathbb{R}, \overline{f}\in(D_{\overline{f}})'\\\\

(2).\overline{f}\,\,\text{es una extesiòn de $f$ a $D_{\overline{f}}$}, i.e, \overline{f}\diagup M=f\\\\

(3). \|\overline{f}\|=\|f\|
\end{cases}$\\\\

O sea que $\forall\overline{f}\in S, D_{\overline{f}}$ es un subespacio de X màs grande que M y que contiene a M.\\\\
Las $\overline{f}\in S$ son A.L.\\
$\begin{diagram}
\node[2]{\overline{f}:M\subset D_{\overline{f}}}\arrow{e,t}\\
\node{\mathbb{R}}
\end{diagram}$\\
$\begin{diagram}
\node[2]{x}\arrow{e,t}\\
\node{\overline{f}(x)}
\end{diagram}$\\\\
Por el lema anterior, $S\neq\emptyset.$\\\\
Vamos a definir en S una $\mathcal{R.O.P}$ asì:\\\\
Sean $\overline{f_1},\overline{f_2}\in S.$ Entonces $\overline{f_1}\preceq\overline{f_2}\Longleftrightarrow D_{\overline{f_1}}\subseteq D_{\overline{f_2}}$ y $\overline{f_2}$ es una extensiòn de $\overline{f_1}$ i.e, $\overline{f_2}\diagup D_{\overline{f_1}}=\overline{f_1}.$\\\\
O sea que $\forall x\in D_{\overline{f_1}},\overline{f_2}(x)=\overline{f_1}(x)$\\\\
Vamos a dm. que $(S,\preceq):$ cjto P.O. o que $\preceq$ es reflexiva,antisimètrica y transitiva.\\\\
\begin{enumerate}
\item $D_{\overline{f_1}}\subset D_{\overline{f_1}}$ y $\overline{f_1}\diagup D_{\overline{f_1}}=\overline{f_1},\overline{f_1}\preceq\overline{f_1}$\\
\item Supongamos que $\overline{f_1}\preceq\overline{f_2}$ y $\overline{f_2}\preceq\overline{f_1}.$ Veamos que: $\overline{f_1}=\overline{f_2}$ o que $D_{\overline{f_1}}=D_{\overline{f_2}}$ y que $\forall x\in D_{\overline{f_1}}=D_{\overline{f_2}}, \overline{f_1}(x)=\overline{f_2}(x).$\\\\
Como \begin{gather}\label{1}\overline{f_1}\preceq\overline{f_2},D_{\overline{f_1}}\subseteq D_{\overline{f_2}}\end{gather}\\\\
y \begin{gather}\label{2}\overline{f_2}\diagup D_{\overline{f_1}}=\overline{f_1} \,\,\text{i.e,} \forall x\in D_{\overline{f_1}}:\overline{f_2}(x)=\overline{f_1}(x)\end{gather}\\\\
Como \begin{gather}\label{3}\overline{f_2}\preceq\overline{f_1},D_{\overline{f_2}}\subseteq D_{\overline{f_1}}\end{gather}\\\\
y $\overline{f_1}\diagup D_{\overline{f_2}}=\overline{f_2}.$\\\\
De ~\eqref{1} y ~\eqref{3}: $D_{\overline{f_1}}=D_{\overline{f_2}}.$\\\\
Segùn ~\eqref{2}, $\forall x\in D_{\overline{f_1}}= D_{\overline{f_2}}:\overline{f_2}(x)=\overline{f_1}(x).$\\\\
Esto dm. que $\overline{f_1}=\overline{f_2}.$\\\\
\item Supongamos ahora que $\overline{f_1}\preceq\overline{f_2}$ y $\overline{f_2}\preceq\overline{f_2}.$ Veamos que $\overline{f_1}\preceq\overline{f_3}.$\\\\
Dos cosas se deben probar: \begin{enumerate}
\item [i)] $D_{\overline{f_1}}\subset \overline{f_3}$\\
\item[ii)] $\overline{f_3}\diagup D_{\overline{f_1}}=\overline{f_1}.$ o que $\forall x\in D_{\overline{f_1}}, \overline{f_3}(x)=\overline{f_1}(x).$\\\\
\end{enumerate}
Como $\overline{f_1}\preceq\overline{f_2},D_{\overline{f_1}}\subseteq D_{\overline{f_2}}\\
\text{como}\,\,\overline{f_2}\preceq\overline{f_3},D_{\overline{f_2}}\subseteq D_{\overline{f_3}}\hspace{0.5cm}\therefore\hspace{0.5cm} D_{\overline{f_1}}\subseteq D_{\overline{f_3}}$ y se tiene i).\\\\
O sea que \begin{gather}\label{4}\forall x\in D_{\overline{f_1}}:\overline{f_2}(x)=\overline{f_1}(x)\end{gather}\\\\
Como $\overline{f_2}\preceq\overline{f_3},\overline{f_3}\diagup D_{\overline{f_2}}=\overline{f_2}$\\\\
i.e que $\forall y\in D_{\overline{f_2}}:\overline{f_3}(y)=\overline{f_2}(y)\hspace{0.5cm}\star$\\\\
Sea $x\in D_{\overline{f_1}}\Longrightarrow x\in D_{\overline{f_2}}\underset{\star}{\Longrightarrow}\overline{f_3}(x)=\overline{f_2}(x).$\\\\
Asì que \begin{gather}\label{5}\forall x\in D_{\overline{f_1}}:\overline{f_3}(x)=\overline{f_2}(x)\end{gather}\\\\
De ~\eqref{4} y ~\eqref{5} se concluye que $\forall x\in D_{\overline{f_1}}:\overline{f_3}(x)=\overline{f_1}(x)$ y se tiene ii).\\\\
Hemos dm. que $(S,\preceq)$ es un cjto. P.O.\\\\
Vìa aplicar el lema de Zorn, sea $\underset{\text{cjto. totalmente ordenado }}{\bigl\{\overline{f_\alpha}\bigr\}_{\alpha\in I}}\subset S.$\\\\
O sea que $\forall\alpha\in I,\overline{f_\alpha}$ es una extensiòn de f. Siendo $\bigl\{\overline{f_\alpha}\bigr\}_{\alpha\in I}$ totalmente ordenado, se tiene que $\forall\alpha_1,\alpha_2\in I,\overline{f_{\alpha_1}}\preceq\overline{f_{\alpha_2}}$ ò $\overline{f_{\alpha_2}}\preceq\overline{f_{\alpha_1}}$.\\\\
Vamos a demostrar que $\bigl\{\overline{f_\alpha}\bigr\}_{\alpha\in I}$ tiene cota superior en S, o que $\exists\overline{f}\in S$ tal que $\forall\alpha\in I:\overline{f_\alpha}\preceq f.$\\\\
Consideremos el cjto. $\underset{\alpha\in I}{\bigcup D}_{\overline{f_\alpha}}\subset X.$\\\\
Veamos que $\underset{\alpha\in I}{\bigcup D}_{\overline{f_\alpha}}$ es un subespacio de X.\\
\begin{enumerate}
\item [i)] Como $\forall\alpha\in I:M\subset D_{\overline{f_\alpha}},M\subset\underset{\alpha\in I}{\bigcup D}_{\overline{f_\alpha}}$ y por lo tanto $\underset{\alpha\in I}{\bigcup D}_{\overline{f_\alpha}}\neq\emptyset.$\\\\
\item[ii)] Sea $x\in\underset{\alpha\in I}{\bigcup D}_{\overline{f_\alpha}},\beta\in \mathbb{R}.$ Veamos que: $\beta x\in\underset{\alpha\in I}{\bigcup D}_{\overline{f_\alpha}}.\\\\ x\in\underset{\alpha\in I}{\bigcup D}_{\overline{f_\alpha}}\Rightarrow\exists\alpha_1\in I$ tal que $x\in D_{\overline{f_{\alpha_1}}}\hspace{0.5cm}\therefore\hspace{0.5cm}(\beta x)\in D_{\overline{f_{\alpha_1}}}$.\\\\
$(\beta x)\in D_{\overline{f_{\alpha_1}}}\Rightarrow(\beta x)\in\underset{\alpha\in I}{\bigcup D}_{\overline{f_\alpha}}.$\\\\
\item[iii)] Sean $x,y\in\underset{\alpha\in I}{\bigcup D}_{\overline{f_\alpha}}$. Veamos que $(x+y)\in\underset{\alpha\in I}{\bigcup D}_{\overline{f_\alpha}}.$\\\\
$x,y\in\underset{\alpha\in I}{\bigcup D}_{\overline{f_\alpha}}\Rightarrow\exists\alpha_1,\alpha_2\in I$ tal que $x\in D_{\overline{f_{\alpha_1}}}, y\in D_{\overline{f_{\alpha_2}}}.$ Pero $\bigl\{\overline{f_\alpha}\bigr\}_{\alpha\in I}\subset S.$ Entonces $\overline{f_{\alpha_1}}\preceq\overline{f_{\alpha_2}}$ ò $\overline{f_{\alpha_2}}\preceq\overline{f_{\alpha_1}},$ i.e, $D_{\overline{f_{\alpha_1}}}\subseteq D_{\overline{f_{\alpha_2}}}$ ò $D_{\overline{f_{\alpha_2}}}\subseteq D_{\overline{f_{\alpha_1}}}$\\\\
Supongamos que se da $D_{\overline{f_{\alpha_1}}}\subseteq D_{\overline{f_{\alpha_2}}}.$ Como $x\in D_{\overline{f_{\alpha_1}}},x\in D_{\overline{f_{\alpha_2}}}, y\in D_{\overline{f_{\alpha_2}}}\Longrightarrow (x+y)\in D_{\overline{f_{\alpha_2}}}\hspace{0.5cm}\therefore\hspace{0.5cm}(x+y)\in\underset{\alpha\in I}{\bigcup D}_{\overline{f_\alpha}}.$\\\\
Si $D_{\overline{f_{\alpha_2}}}\subseteq D_{\overline{f_{\alpha_1}}},\cdots\cdots$\\\\
Esto demuestra que $\underset{\alpha\in I}{\bigcup D}_{\overline{f_\alpha}}$ es un subespacio de X. Pasemos ahora a definir a $\overline{f}.$\\\\
Tomemos $x\in\underset{\alpha\in I}{\bigcup D}_{\overline{f_\alpha}}.$ Entonces $\exists\beta\in I$ tal que $x\in\underset{\alpha\in I}{\bigcup D}_{\overline{f_\beta}}\Longrightarrow\overline{f_\beta}(x)\in\mathbb{R}.$\\\\
Definamos \\
$\begin{diagram}
\node[2]{\overline{f}:\underset{\alpha\in I}{\bigcup D}_{\overline{f_\alpha}}}\arrow{e,t}\\
\node{\mathbb{R}}
\end{diagram}$\\
$\begin{diagram}
\node[2]{x}\arrow{e,t}\\
\node{\overline{f}(x)=\overline{f}_\beta(x)}
\end{diagram}$\\\\
Debemos dm. que $\overline{f}$ està bièn definida.\\\\
Supongamos que $\exists\alpha,\beta\in I$ tal que $x\in D_{\overline{f_\alpha}}$ y $x\in D_{\overline{f_\beta}}.$\\\\
Para establecer que $\overline{f}$ està bien definida debemos dm. que \begin{gather}\label{6}\overline{f_\alpha}(x)=\overline{f_\beta}(x)\end{gather}\\\\
Como $\alpha,\beta\in I,\overline{f_\alpha}\preceq\overline{f_\beta}$ ò $\overline{f_\beta}\preceq\overline{f_\alpha}.$\\\\
Veamos que en cualquier caso, $\overline{f_\alpha}(x)=\overline{f_\beta}(x).$\\\\
Suponagmos que $\overline{f_\alpha}\preceq\overline{f_\beta}.$\\\\
Entonces $D_{\overline{f_\alpha}}\subseteq D_{\overline{f_\beta}}$ y $\overline{f_\beta}\diagup D_{\overline{f_\alpha}}=\overline{f_\alpha}.$\\\\
O sea que $\forall y\in D_{\overline{f_\alpha}}:\overline{f_\beta}(y)=\overline{f_\alpha}(y).$\\\\
Como $x\in D_{\overline{f_\alpha}},\overline{f_\beta}(x)=\overline{f_\alpha}(x)$ y se tiene ~\eqref{6}.\\\\
Si se dice que $\overline{f_\beta}\preceq\overline{f_\alpha}$ se procede de manera anàloga. Esto dm. $\overline{f}$ està bien definida.\\\\
Veamos ahora que $\overline{f}\in S$ o que\\
\item[i)] $\overline{f}\diagup M=f$\\
\item[ii)] $\overline{f}\in\left(\underset{\alpha\in I}{\bigcup D}_{\overline{f_\alpha}}\right)'$\\
\item[iii)] $\|\overline{f}\|=\|f\|.$\\\\
\item Sea $x\in M.$ Veamos que: $\overline{f}=f(x).$\\\\
Como $x\in M\subset\underset{\alpha\in I}{\bigcup D}_{\overline{f_\alpha}},\exists\beta\in I$ tal que $x\in D_{\overline{f_\beta}}\Longrightarrow\overline{f}(x)=\overline{f_\beta}(x)=f(x)\\\\
\begin{cases}
S\in\overline{f_\beta} \,\,\text{y por lo tanto},\\\\
\overline{f_\beta}\diagup M=f.\,\,\text{Como} x\in M,\overline{f_\beta}(x)=f(x).
\end{cases}$\\\\
\item[ii)] Veamos $\overline{f}\in\left(\underset{\alpha\in I}{\bigcup D}_{\overline{f_\alpha}}\right)'.$ Primero veamos que ${\overline{f}\in\left(\underset{\alpha\in I}{\bigcup D}_{\overline{f_\alpha}}\right)}^{*},$ i.e, que f es A.L.\\\\
Sean $x,y\in\underset{\alpha\in I}{\bigcup D}_{\overline{f_\alpha}}\Longrightarrow\exists\beta,\varpi\in I$ tal que $x\in D_{\overline{f_\beta}}$ y $y\in D_{\overline{f_\varpi}}.$ Entonces $x+y\in\underset{\alpha\in I}{\bigcup D}_{\overline{f_\alpha}}$ ya que es un subespacio de X. Veamos que $\overline{f}(x+y)=\overline{f}(x)+\overline{f}(y)$.\\\\
Pero $\overline{f_\beta}\text{y $\overline{f_\varpi}$}\in\bigl\{\overline{f_\alpha}\bigr\}_{\alpha\in I}\subset S.$\\\\
Luego $D_{\overline{f_\beta}}\subseteq D_{\overline{f_\varpi}}$ ò $D_{\overline{f_\varpi}}\subseteq D_{\overline{f_\beta}}.$\\\\
Suponganos que $D_{\overline{f_\beta}}\subseteq D_{\overline{f_\varpi}}.$\\\\
Como $x\in D_{\overline{f_\beta}},x\in D_{\overline{f_\varpi}};y\in D_{\overline{f_\varpi}}\Longrightarrow(x+y)\in D_{\overline{f_\varpi}}.$\\\\
$\overline{f}(x+y)=\overline{f_\varpi}(x+y).\,\,{\overline{f_\varpi}\in\left(D_{\overline{f_\varpi}}\right)}^{*}=\overline{f_\varpi}(x)+\overline{f_\varpi}(y)=f(x)+f(y).$\\\\
Sea $x\in\underset{\alpha\in I}{\bigcup D}_{\overline{f_\alpha}}$ y $\lambda\in\mathbb{R}.\,\,x\in\underset{\alpha\in I}{\bigcup D}_{\overline{f_\alpha}}\Longrightarrow\exists\beta\in I$ tal que $x\in  D_{\overline{f_\beta}}:$ subespacio de X.$\hspace{0.5cm}\therefore\hspace{0.5cm}\Longrightarrow\lambda x\in D_{\overline{f_\beta}}\Longrightarrow\overline{f}(\lambda x)=\overline{f_\beta}(\lambda x)=\lambda\overline{f_\beta}(x)=\lambda\overline{f}(x).$\\\\
Veamos $\overline{f}$ es continua. Sea $x\in\underset{\alpha\in I}{\bigcup D}_{\overline{f_\alpha}}.$\\\\
Entonces $\exists\beta\in I$ tal que $x\in D_{\overline{f_\beta}}\Longrightarrow\overline{f}(x)\in\mathbb{R}$ y por la def. de $\overline{f}, f(x)=\overline{f_\beta}(x).\hspace{0.5cm}\therefore\hspace{0.5cm}\left|\overline{f}(x)\right|=\left|\overline{f_\beta}(x)\right|\leqslant\|\overline{f_\beta}\|\|x\|=\|f\|\|x\|$\\\\
Como $\overline{f_\beta}\in S,\|\overline{f_\beta}\|=\|f\|.$\\\\
Asì que $\forall x\in\underset{\alpha\in I}{\bigcup D}_{\overline{f_\alpha}},\left|\overline{f}(x)\right|\leqslant\|f\|\|x\|$ lo que dm. que $\overline{f}$ es continua.\\\\
Veamos ahora que $\|\overline{f}\|=\|f\|.$\\\\
$\|\overline{f}\|=\underset{\|x\|\leqslant1}{\sup\left|\overline{f}(x)\right|}.$\\\\
Sea $x\in\underset{\alpha\in I}{\bigcup D}_{\overline{f_\alpha}}$ con $\|x\|\leqslant1.$ Entonces, por lo establecido antes, $\exists\beta\in I$ tal que $\left|\overline{f}(x)\right|\leqslant\|f\|\|x\|\leqslant\|f\|.$\\\\
Esto dm. que el real $\|f\|$ es cota superior del cjto. $\left\{\left|\overline{f}(x)\right|,x\in\underset{\alpha\in I}{\bigcup D}_{\overline{f_\alpha}},\|x\|\leqslant1\right\}.$\\\\
Luego $\underset{\|x\|\leqslant1}{\sup\left|\overline{f}(x)\right|}\leqslant\|f\|$, i.e, $\|\overline{f}|\leqslant\|f\|.$\\\\
Veamos que $\|f\|\leqslant\|\overline{f}\|.$\\\\
Batarà con dm. que \begin{gather}\label{6}\left\{\left|\overline{f}(x)\right|,x\in M,\|x\|\leqslant1\right\}\subset\left\{\left|\overline{f}(x)\right|,x\in\underset{\alpha\in I}{\bigcup D}_{\overline{f_\alpha}},\|x\|\leqslant1\right\}\end{gather}\\\\
ya que el tomar el supremo se tendrà que $\underset{x\in M,\|x\|\leqslant1}{\sup\left|f(x)\right|}\leqslant\underset{x\in\underset{\alpha\in I}{\bigcup D}_{\overline{f_\alpha}}}{\sup\left|\overline{f}(x)\right|}$\\\\
i.e, $\|f\|\leqslant\|\overline{f}\|.$\\\\
Establezcamos pues, ~\eqref{6}.\\\\
Sea $y$ un vector en el primer cjto.\\\\
Entonces $\exists x\in M\Longrightarrow x\in\underset{\alpha\in I}{\bigcup D}_{\overline{f_\alpha}}$ con $\|x\|\leqslant1$ y tal que $y=\|f(x)\|$.\,\,$x\in\underset{\alpha\in I}{\bigcup D}_{\overline{f_\alpha}}\Longrightarrow\exists\beta\in I$ tal que $x\in D_{\overline{f_\beta}}\Longrightarrow\overline{f}(x)=\overline{f_\beta}(x)\Longrightarrow f(x)=\overline{f}(x).$ Pero $y=\left|f(x)\right|.$ Entonces $y=\left|\overline{f}(x)\right|.$\\\\
$\exists\beta\in I\Longrightarrow\overline{f_\beta}\in S.\hspace{0.5cm}\therefore\hspace{0.5cm}\overline{f_\beta}\diagup M=f.$ Como $x\in M,\overline{f_\beta}=f(x).$\\\\
Asì que $\exists x\in\underset{\alpha\in I}{\bigcup D}_{\overline{f_\alpha}}$ con $\|x\|\leqslant1$ tal que $y=\left|\overline{f}(x)\right|,$ lo que dm. que $y$ està en el $2^{do}$ cjto.\\\\
Veamos que $\forall\alpha\in I:\overline{f_\alpha}\preceq\overline{f}$ con lo que habremos dm. que $\left\{\overline{f_\alpha}\right\}_{\alpha\in I}$ tiene una cota superior en S.\\\\
Fijemos $\alpha\in I.$ Veamos que $\overline{f_\alpha}\preceq\overline{f}.$\\\\
Se debe probar que $D_{\overline{f_\beta}}\subseteq D_{\overline{f}}$ (esto es claro). y que $\overline{f}\diagup D_{\overline{f}}=\overline{f_\alpha},$ o que $\forall x\in D_{\overline{f_\alpha}},\overline{f}=\overline{f_\alpha}.$\\\\
Tomemos $x\in D_{\overline{f_\beta}}.$ Entonces, por la def. de $\overline{f},\overline{f}(x)=\overline{f_\alpha}(x).$\\\\
Hemos dm. hasta aquì que\\
$\left\{(S,\preceq):\text{cjto. P.O. y que} \forall\left(\overline{f_\alpha}\right)_{\alpha\in I},\\\\
\left\{\overline{f_\alpha}\right\}_{\alpha\in I}\,\,\text{tiene cota superior en S}.\right\}$\\\\
Luego, por el $\textit{Lema de Zorn,}$ S tiene elemento maximal,\,\,i.e,$\exists D_F:$ sub.vect. de X con $M\subset D_F$ y $\exists F\in(D_F)'$ i.e,\,\, $F:D_F\longrightarrow\mathbb{R}\diagup\text{A.L. Continua}$ tal que\\\\
$F\diagup M=f\\
\|F\|=\|f\|$\\\\
y $\nexists\overline{f}\in S$ tal que \begin{gather}\label{7}F\preceq\overline{f}\end{gather} o lo que es lo mismo, $\nexists\overline{f}\in S$ tal que $D_F\subset D_f$ y $\overline{f}\diagup D_F=F.$\\\\
Para completar la prueba del T.H.B. solo resta dm. que $D_F=X.$\\\\
Razonemos por R.Abs.\\\\
Supongamos que no es asì. Entonces $\exists x_0\in X,x_0 \notin D_F$ y por el $\text{Lema},\exists\overline{f}\in\left(Sg(D_F)\cup\bigl\{x_0\bigr\}\right)'=(D_F)'$ tal que $\overline{f}\diagup D_F=F$ y $\|\overline{f}\|=\|F\|.$\\\\
Veamos que $\overline{f}\in S$ y que $F\preceq\overline{f}$\\\\
lo que constituye una $(\rightarrow\leftarrow)$ con ~\eqref{7}\\\\
De esta manera $X=D_F$ y termina la dm.\\\\
Establezcamos pues ~\eqref{7}. Es claro que \begin{gather*}\overline{f}\in(D_F)'\\
M\subset D_F\\
\|\overline{f}\|=\|F\|=\|f\|\end{gather*}\\\\
Veamos que $\overline{f}\diagup M=f$ o que $\forall x\in M:\overline{f}(x)=f(x).$\\\\
$\forall x\in M\Longrightarrow F\diagup M=f. F(x)=f(x).$ Como $x\in M\subset D_F$ y $\overline{f}\diagup D_F=F,\overline{f}(x)=F(x).$\\\\
Esto dm. que $\overline{f}\in S.$\\\\
Finalmente, como $D_F\subset D_{\overline{f}}$ y $\overline{f}\diagup D_F=F$ se tiene que $F\preceq f$ y se tiene ~\eqref{7}.\\\\
\end{enumerate}
\end{enumerate}
\end{proof}
\section{$\mathbf{\text{Algunas consecuencias del T.H.B}}$}
Recordemos el enunciado del T.H.B.\\\\
"todo funcional continuo definido sobre un subespacio vectorial de un E.L.N se puede extender (prolongar) a todo el espacio conservando la norma":\\\\
$(X,\|\|):\mathbb{R}$ Esp. vec. normado. $\exists F\in X'=\mathcal{L}_c(X,\mathbb{R})$ tal que\\
\begin{enumerate}
\item $F\diagup M=f$\\
\item $\|F\|=\|f\|$\\\\
\end{enumerate}
ò lo que es lo mismo, $\underset{x\in X;\|x\|\leqslant1}{\sup\left|F(x)\right|}=\underset{x\in X;\|x\|\leqslant1}{\sup\left|f(x)\right|}$\\\\
\begin{con}
Sea $(X,\|\|):\mathbb{R}$ esp. vect. normado y sea $\underset{fijo}{\underbrace{x}}\in X, x\neq0.$ Entonces $\exists x'\in X$ tal que\\
\begin{enumerate}
\item $\|x'\|=1$\\
\item $\langle x,x'\rangle=\|x\|$\\
\end{enumerate}
\end{con}
\begin{proof}
Definamos\\

$\begin{diagram}
\node[2]{f:Sg\bigl\{x\bigr\}}\arrow{e,t}\\
\node{\mathbb{R}}
\end{diagram}$\\
$\begin{diagram}
\node[2]{\alpha x}\arrow{e,t}\\
\node{f(\alpha x)=\alpha\|x\|}
\end{diagram}$\\\\

Veamos que $f\in\left(Sg\bigl\{x\bigr\}\right)'=\mathcal{L}_c(Sg\bigl\{x\bigr\},\mathbb{R}).$\\\\
Tomemos $\alpha_1 x,\alpha_2 x\in Sg\bigl\{x\bigr\}.$\\\\
Entonces $f\left(\alpha_1 x+\alpha_2 x\right)=f\left((\alpha_1 +\alpha_2)\right)x\\\\
=(\alpha_1+\alpha_2)\|x\|=\alpha_1\|x\|+\alpha_2\|x\|=f(\alpha_1,x)+f(\alpha_2,x).$\\\\
Sea $\lambda\in\mathbb{R},(\alpha x)\in Sg\bigl\{x\bigr\}.$\\\\
$f\left(\lambda(\alpha x)\right)=f\left((\lambda\alpha)x\right)=\lambda\alpha\|x\|=\lambda f(\alpha x).$\\\\
Esto dm. que f es A.L.\\\\
Ahora, $\forall(\alpha x)\in Sg\bigl\{x\bigr\}:\\\\
\left|f(\alpha x)\right|=\left|\alpha\|x\|\right|=|\alpha|\|x\|=\|\alpha x\|=1.\|\alpha x \|,$ lo que dm. que f es continua y de este modo se tiene que $f\in\left(Sg\bigl\{x\bigr\}\right)'.$\\\\
Luego, por el T.H.B, $\exists x'\in X'$ i.e\\
$\begin{diagram}
\node[2]{\exists x':X}\arrow{e,t}\\
\node{\mathbb{R}}
\end{diagram}$\\
$\begin{diagram}
\node[2]{y}\arrow{e,t}\\
\node{\langle y,x'\rangle}
\end{diagram}$\\\\
tal que\\
\begin{enumerate}
\item $x'\diagup Sg\bigl\{x\bigr\}=f$\\
\item $\|x'\|=\|f\|$\\\\
\end{enumerate}
Segùn (1), $\forall\alpha x\in Sg\bigl\{x\bigr\}:\langle \alpha x,x'\rangle=f(\alpha x).$\\\\
Luego si $\alpha=1,\langle x,x'\rangle=f(1.x)=f(x)=\langle x,f\rangle=\|x\|$\\\\
De esta manera hemos dm que $\exists x'\in X'$ tal que $\langle x,x'\rangle=\|x\|.$\\\\
Veamos ahora que $\|x'\|=1\hspace{0.5cm}\star.$\\\\
Segùn (2), $\|x'\|=\|f\|.$ Luego para tener $\star$ veamos que $\|f\|=1.$\\\\
$\|f\|=\underset{\|\alpha x\|\leqslant1}{\sup\left|\langle \alpha x,f\rangle\right|}=\sup|\alpha|\left|\langle x,f\rangle\right|=\|x\|\sup|\alpha|=\|x\|\dfrac{1}{\|x\|}=1$\\
\end{proof}
\begin{con}
Sea $(X,\|\|):\mathbb{R}$ esp. vect. Normado, $F\subset X,F$ cerrado y $x\in X;x\notin F.$\\\\
$\exists x'\in X'$ tal que\\
\begin{enumerate}
\item $\langle x,x'\rangle=1$\\
\item $\forall\in F:\langle y,x'\rangle=0.$\\\\
\end{enumerate}
O sea que "si $F\subset(X,\|\|)$ y $x\notin F,$ hay un funcional continuo en X que vale 1 en x y se anula en F".
\end{con}
\begin{proof}
Sea $M=F+Sg\bigl\{x\bigr\}=\left\{y+\alpha x, y\in F,\alpha\in\mathbb{R}\right\}$\\\\
Es claro que M es un subespacio vect. de X.\\\\
Sea \\
$\begin{diagram}
\node[2]{f:M=F+Sg\bigl\{x\bigr\}=\left\{y+\alpha x, y\in F,\alpha\in\mathbb{R}\right\}}\arrow{e,t}\\
\node{\mathbb{R}}
\end{diagram}$\\
$\begin{diagram}
\node[3]{y+\alpha x}\arrow{e,t}\\
\node{f(y+\alpha x)=\alpha}
\end{diagram}$\\\\
f es A.L. ya que $f\left((y_1+\alpha_1 x)+(y_2+\alpha_2 x)\right)=f\left((y_1+y_2)+(\alpha_1+\alpha_2)x\right)\\
=\alpha_1+\alpha_2=f(y_1+\alpha_1 x)+f(y_1+\alpha_2x).\\\\
f\left(\lambda(y+\alpha x)\right)=f\left(\lambda y+(\lambda\alpha)x\right)=\lambda\alpha=\lambda f(y+\alpha x)$\\\\
Esto dm. que f es A.L.\\\\
Nòtese a demàs que de la def. de f,\begin{equation}\begin{split}\forall y\in F:f(y)=f(y+0.x)=0\\\\
\text{Ademàs},f(x)=f(0+1.x)=1\end{split}\end{equation}\\\\
Recordemos ahora el sgte. resultado de la topologìa:\\
$\begin{cases}
Sea A\subset(X,\|\|)\,\,\text{y$x\in X$}.\\
x\in\overline{A}\Longleftrightarrow d(x,A)=\underset{y\in A}{\inf\|x-y\|=0}
\end{cases}$\\\\
En nuestro caso $F\subset(X,\|\|)$ y $x\notin F.$ Luego $x\notin \overline{F}$ y por lo tanto $d(x,F)=\underset{y\in F}{\inf\|x-y\|}>0.$\\\\
O sea que \begin{gather}\label{8}\forall y\in F:\|x-y\|\geqslant d(x,F)>0\end{gather}\\\\
Tomemos $(y-\alpha x)\in M$ con $\alpha\neq0.$ Entones $\|y-\alpha x\|=\left\|\alpha\left(\frac{1}{\alpha}y-x\right)\right\|\\\\
=|\alpha|\left\|\frac{1}{\alpha}y-x\right\|\underset{~\eqref{8}}{\geqslant}|\alpha|\alpha(x,F)$\\\\
O sea que $|\alpha|\leqslant\dfrac{\|y-\alpha x\|}{d(x,F)}.$\\
\end{proof}
\section{Los espacios $\mathfrak{L}^p$}
\begin{defin}
Sea $\underset{fijo}{p}\geqslant 1.$ El espacio $\mathfrak{L}^p$ se define como $$\mathfrak{L}^p=\left\{x=(\xi_j)_{j=1}^{\infty}\in S(\mathbb{K})\,\,\text{tal que la serie de nùmeros reales}\,\,|\xi_1|^p+|\xi_2|^p+\ldots\ldots+\text{es convergente}\right\}$$\\\\
$=\left\{x=(\xi_j)_{j=1}^{\infty}\in S(\mathbb{K})\,\,\text{tal que}\sum\limits_{j=1}^{\infty}|\xi_j|^p<\infty\right\}$\\
\end{defin}
\begin{prop}
$\mathfrak{L}^p$ es un subespacio de $S(\mathbb{K})$.\\
\end{prop}
\begin{proof}
\begin{enumerate}
\item Veamos que $\mathfrak{L}^p\neq\emptyset.$\\
Consideremos una sucesiòn cualquiera que tenga una cola de  infinitos ceros:\\$x=\left\{\xi_1,\xi_2,\ldots,\ldots,\xi_n,0,0,\ldots\ldots\right\}$\\
Es claro que $|\xi_1|^p+|\xi_2|^p+\ldots\ldots+$ converge a $\sum\limits_{j=1}^n|\xi_j|^p$ y por lo tanto $\mathfrak{L}^p\neq\emptyset.$\\
Otra forma de verlo es la siguiente:$\forall p>1,$ si $x=\left\{\dfrac{1}{j}\right\}_{j=1}^\infty\in S(\mathbb{R}),$ la serie $1+\frac{1}{2^p}+\frac{1}{3^p}+\ldots\ldots+$ converge. Luego $x=\left\{\dfrac{1}{j}\right\}_{j=1}^\infty\in\mathfrak{L}^p$ lo que dm. una vez màs que $\mathfrak{L}^p\neq\emptyset.$\\\\
\item Sea $x=\{\xi_j\}_{j=1}^\infty,y=\{\eta_j\}_{j=1}^\infty\in\mathfrak{L}^p,$ o sea que las sucesiones $|\xi_1|^p+|\xi_2|^p+\ldots\ldots+\,\,\text{converge a}\,\,\sum\limits_{i=1}^\infty|\xi_i|^p\\
|\eta_1|^p+|\eta_2|^p+\ldots\ldots+\,\,\text{converge a}\,\,\sum\limits_{i=1}^\infty|\eta_i|^p$\\
Veamos que $(x+y)\in\mathfrak{L}^p$ o que la serie $|\xi_1+\eta_1|^p+|\xi_2+\eta_2|^p+\ldots\ldots+\sum\limits_{i=1}^\infty|\xi_i+\eta_i|^p$\\\\
Fijemos n en $\mathbb{N}$.\\
Por la Desigualdad de Minkowski en el  ELN$(\mathbb{K}^n,\|\|_p)\diagup\|x+y\|_p\leqslant\|x\|_p+\|y\|_p$\\
\begin{gather}\label{9}\left(\sum\limits_{i=1}^n|\xi_i+\eta_i|^p\right)^{1/p}\leqslant\left(\sum\limits_{i=1}^n|\xi_i|^p\right)^{1/p}+\left(\sum\limits_{i=1}^n|\eta_i|^p\right)^{1/p}\end{gather}\\
Pero como $x\in\mathfrak{L}^p,\sum\limits_{i=1}^n|\xi_i|^p\leqslant\sum\limits_{i=1}^\infty|\xi_i|^p$\\\\
y como $y\in\mathfrak{L}^p,\sum\limits_{i=1}^n|\eta_i|^p\leqslant\sum\limits_{i=1}^\infty|\eta_i|^p$\\\\
$\Longrightarrow\left(\sum\limits_{i=1}^n|\xi_i|^p\right)^{1/p}+\left(\sum\limits_{i=1}^n|\eta_i|^p\right)^{1/p}\leqslant\left(\sum\limits_{i=1}^\infty|\xi_i|^p\right)^{1/p}+\left(\sum\limits_{i=1}^\infty|\eta_i|^p\right)^{1/p}$\\\\\\
que en ~\eqref{9} $\Longrightarrow\left(\sum\limits_{i=1}^n|\xi_i+\eta_i|^p\right)^{1/p}\leqslant\left(\sum\limits_{i=1}^n|\xi_i|^p\right)^{1/p}+\left(\sum\limits_{i=1}^n|\eta_i|^p\right)^{1/p}\\\\\\
\Longrightarrow\sum\limits_{i=1}^n|\xi_1+\eta_i|^p\leqslant{\left[\left(\sum\limits_{i=1}^\infty|\xi_i|^p\right)^{1/p}+\left(\sum\limits_{i=1}^\infty|\eta_i|^p\right)^{1/p}\right]}^p$\\\\
O sea que la sucesiòn de S. parciales de la de tèrminos positivos $|\xi_1+\eta_1|^p+|\xi_2+\eta_2|^p+\ldots\ldots+$ tiene una cota superior y por tanto converge.\\\\
\item tomemos $\lambda\in\mathbb{K}$ y $x=\bigl\{\xi_j\bigr\}_{j=1}^\infty\in\mathfrak{L}^p.$ Veamos que $(\lambda x)=\bigl\{\xi_j\bigr\}_{j=1}^\infty\in\mathfrak{L}^p.$\\
Debemos dm. que $|\lambda\xi_1|^p+|\lambda\xi_2|^p+\ldots\ldots+\sum\limits_{j=1}^\infty|\lambda\xi_i|^p$\\
Como $x=\bigl\{\xi_j\bigr\}_{j=1}^\infty\in\mathfrak{L}^p,|\xi_1|^p+|\xi_2|^p+\ldots\ldots+\sum\limits_{j=1}^\infty|\xi_i|^p\\
\Longrightarrow|\lambda\xi_1|^p+|\lambda\xi_2|^p+\ldots\ldots+|\lambda|^p\sum\limits_{j=1}^\infty|\xi_i|^p$\\
Luego, $\mathfrak{L}^p$ es un subespacio de $S(\mathbb{K})$ y por lo tanto $\mathfrak{L}^p$ es un $\mathbb{K}$ Esp. Vect.\\
\end{enumerate}
\end{proof}
\begin{prop}
Sea $p\geqslant1.$ La funciòn\\
$\begin{diagram}
\node[2]{\|\|_p:\mathfrak{L}^p}\arrow{e,t}\\
\node{\mathbb{R}}
\end{diagram}$\\
$\begin{diagram}
\node[2]{x=\bigl\{\xi_j\bigr\}_{j=1}^\infty}\arrow{e,t}\\
\node{{\|x\|_p=\left(\sum\limits_{j=1}^\infty|\xi_j|^p\right)}^{1/p}}
\end{diagram}$\\
es una norma en $\mathfrak{L}^p$ y por tanto $(\mathfrak{L}^p,\|\|_p):$ E.L.N\\
\end{prop}
\begin{proof}
\begin{enumerate}
\item Es claro que $\|x\|_p\geqslant0.$ Supongamos que $\|x\|_p=0$ entonces ${\left(\sum\limits_{j=1}^\infty|\xi_j|^p\right)}^{1/p}=0\Longrightarrow\bigl\{\xi_j\bigr\}=0\,\,\forall j,$ i.e, x es la sucesiòn cero.\\
\item Sea $x=\bigl\{\xi_j\bigr\}_{j=1}^\infty\in\mathfrak{L}^p,\alpha\in\mathbb{K}.$ Entonces $\alpha x=\bigl\{\alpha\xi_j\bigr\}_{j=1}^\infty\in\mathfrak{L}^p\Longrightarrow\sum\limits_{j=1}^\infty|\alpha\xi_j|^p$ converge a $|\alpha|\sum\limits_{j=1}^\infty|\xi_i|^p.\\\\
\|\alpha x\|={\left(\sum\limits_{j=1}^\infty|\alpha\xi_j|^p\right)}^{1/2}=|\alpha|{\left(\sum\limits_{j=1}^\infty|\xi_j|^p\right)}^{1/2}=|\alpha|\|x\|_p$\\
\item Supongamos $p=1$ y sean $x,y\in\mathfrak{L}^1.$ Veamos que $\|x+y\|_p\leqslant\|x\|_p+\|y\|_p.$\\\\
Entonces $|\xi_1|+|\xi_2|+\ldots+\ldots+$ converge a $\bigl\{x_j\bigr\}_{j=1}^\infty\\\\
|\eta_1|+|\eta_2|+\ldots\ldots+$ converge a $\bigl\{\eta_j\bigr\}$\\\\
Por la desigualdad triangular, $|\xi_j+\eta_j|\leqslant|\xi_j|+|\eta_j|$\\\\
Como $|\xi_1|+|\xi_2|+\ldots+\ldots$ converge y $|\eta_1|+|\eta_2|+\ldots\ldots+$ converge, la serie $(|\xi_1|+|\eta_1|)+(|\xi_2|+|\eta_2|)+\ldots\ldots+$ converge a $\left(\sum\limits_{j=1}^\infty|\xi_j|+\sum\limits_{j=1}^\infty|\eta_j|\right)$\\\\
Asì que $(|\xi_1|+|\eta_1|)+(|\xi_2|+|\eta_2|)+\ldots\ldots+$ converge y $|\xi_j+\eta_j|\leqslant|\xi_j|+|\eta_j|$ Luego por el criterio de comparaciòn, $|\xi_1+\eta_1|+|\xi_2+\eta_2|+\ldots\ldots$ converge y su suma es tal que:$\sum\limits_{j=i}^\infty|\xi_j+\eta_j|\leqslant\sum\limits_{j=1}^\infty|\xi_j|+\sum\limits_{j=1}^\infty|\eta_j|$ y se tiene asì la desigualdad $\Delta_r$ para $p=1.$\\\\
Sea ahora $p>1$ y consideremos las sucesiones $x,y\in\mathfrak{L}^p.$ Entonces $\bigl\{x+y\bigr\}\in\mathfrak{L}^p.$ Veamos que $\|x+y\|_p\leqslant\|x\|_p+\|y\|_p.$\\\\
Fijemos $n\in\mathbb{N}.$\\ Como $x=\bigl\{\xi_j\bigr\}_{j=1}^\infty\in\mathfrak{L}^p,\sum\limits_{j=1}^n|\xi|^p\leqslant\sum\limits_{j=1}^\infty|\xi|^p\Longrightarrow{\left(\sum\limits_{j=1}^n|\xi|^p\right)}^{1/2}\leqslant{\left(\sum\limits_{j=1}^n|\xi|^p\right)}^{1/2}=\|x\|_p$\\\\
Como $y=\bigl\{\eta_j\bigr\}_{j=1}^\infty\in\mathfrak{L}^p,\sum\limits_{j=1}^n|\eta|^p\leqslant\sum\limits_{j=1}^\infty|\eta|^p\Longrightarrow{\left(\sum\limits_{j=1}^n|\eta|^p\right)}^{1/2}\leqslant{\left(\sum\limits_{j=1}^n|\eta|^p\right)}^{1/2}=\|y\|_p$\\\\
Ahora por la Des. de Minkowski: ${\left(\sum\limits_{j=1}^n|\xi_i+\eta_j|^p\right)}^{1/2}\leqslant{\left(\sum\limits_{j=1}^\infty|\xi_j|\right)}^{1/2}+{\left(\sum\limits_{j=1}^\infty|\eta_j|\right)}^{1/2}\leqslant\|x\|_p+\|y\|_p$\\\\
e.i, $\sum\limits_{j=1}^n|\xi_i+\eta_i|^p\leqslant{\left(\|x\|_p+\|y\|_p\right)}^p$ i.e, la suc. de sumas parciales de la serie $|\xi_1+\eta_1|^p+|\xi_2+\eta_2|^p+\ldots\ldots+ $ està acotada superiormente. La serie $|\xi_1+\eta_1|^p+|\xi_2+\eta_2|^p+\ldots\ldots+$ converge y $\sum\limits_{j=1}^\infty|\xi_i+\eta_i|^p\leqslant{\left(\|x\|_p+\|y\|_p\right)}^{1/2}$ i.e $\|x+y\|_p\leqslant\|x\|_p+\|y\|_p$\\\\
\end{enumerate}
\end{proof}
\begin{prop}[$\textbf{\text{Desigualdad de H\"{o}lder en los $\mathfrak{L}^p$}}$ ]
Sean p,q exponentes conjungados, y sean $x=\bigl\{\xi_j\bigr\}_{j=1}^\infty\in\mathfrak{L}^p, y=\bigl\{\eta_j\bigr\}_{j=1}^\infty\in\mathfrak{L}^q.$\\\\
Entonces la serie $|\xi_1\eta_1|+|\xi_2\eta_2|+\ldots\ldots$ converge y $\sum\limits_{j=1}^\infty|\xi_l\eta_j|\leqslant\|x\|_p\|y\|_q.$\\
\end{prop}
\begin{proof}
Como $x=\bigl\{\xi_j\bigr\}_{j=1}^\infty\in\mathfrak{L}^p,y=\bigl\{\eta_j\bigr\}_{j=1}^\infty\in\mathfrak{L}^q,\sum\limits_{j=1}^\infty|\xi_j|^p<\infty,\sum\limits_{j=1}^\infty|\eta_j|^q<\infty.$\\\\
Definamos las sucesiones:\\
\begin{gather}\label{10}\overline{\xi_j}=\dfrac{\xi_j}{{\left(\sum\limits_{k=1}^\infty|\xi_k|^p\right)}^{1/p}}\hspace{0.5cm} ,\hspace{0.5cm}\overline{\eta_j}=\dfrac{\eta_j}{{\left(\sum\limits_{k=1}^\infty|\eta_k|^q\right)}^{1/q}}\end{gather}\\\\
Entonces $|\overline{\xi_j}|^p=\dfrac{|\xi_j|^p}{\sum\limits_{k=1}^\infty|\xi_k|^p}\hspace{0.5cm},\hspace{0.5cm}|\overline{\eta_j}|^q=\dfrac{|\eta_j|^q}{\sum\limits_{k=1}^\infty|\eta_k|^q}$\\\\
Y las series $|\overline{\xi_1}|^p+|\overline{\xi_2}|^p+\ldots\ldots\\\\
|\overline{\eta_1}|^p+|\overline{\eta_2}|^p+\ldots\ldots$ convergen a 1. En efecto, $|\overline{\xi_1}|^p+|\overline{\xi_2}|^p=\dfrac{1}{\sum\limits_{k=1}^\infty|\xi_k|^p}\left(|\xi_1|^p+|\xi_2|^p+\ldots\ldots\right)=\dfrac{\sum\limits_{k=1}^\infty|\xi_k|^p}{\sum\limits_{k=1}^\infty|\xi_k|^p}=1\\\\
|\overline{\eta_1}|^q+|\overline{\eta_2}|^q\dfrac{1}{\sum\limits_{k=1}^\infty|\eta_k|^q}\left(|\eta_1|^q+|\eta_2|^p+\ldots\ldots\right)=\dfrac{\sum\limits_{k=1}^\infty|\eta_k|^q}{\sum\limits_{k=1}^\infty|\eta_k|^q}=1$\\\\
Ahora $|\overline{\xi_j}\overline{\eta_j}|\leqslant\dfrac{|\overline{\eta_j}|^p}{p}+\dfrac{|\overline{\eta_j}|^q}{q}$\\\\
Como la serie $\dfrac{1}{p}\sum\limits_{j=1}^\infty|\overline{\xi_j}|^p$ converge a $\dfrac{1}{p};$ y como la serie $\dfrac{1}{q}\sum\limits_{j=1}^\infty|\overline{\eta_j}|^q,$ converge a $\dfrac{1}{q}$ entonces por el criterio de comparaciòn, la serie $\sum\limits_{j=1}^\infty|\overline{\xi_j}\overline{\eta_j}|$ converge, y se tiene que su suma està acotada $\sum\limits_{j=1}^\infty|\overline{\xi_j}\overline{\eta_j}|\leqslant\dfrac{1}{p}\sum\limits_{j=1}^\infty|\overline{\xi_j}|^p+\dfrac{1}{q}\sum\limits_{j=1}^\infty|\overline{\eta_j}|^q=\dfrac{1}{p}+\dfrac{1}{q}=1$\\\\
y si se tiene encuenta ~\eqref{10}, $\dfrac{\sum\limits_{j=1}^\infty|\overline{\xi_j}\overline{\eta_j}|}{\|x\|_p.\|y\|_q}\leqslant1\Longrightarrow\sum\limits_{j=1}^\infty|\overline{\xi_j}\overline{\eta_j}|\leqslant\|x\|_p\|y\|_q$\\

\end{proof}
\begin{prop}
$\mathfrak{L}^2$ es completo y por lo tanto $\mathfrak{L}^2:$ E. Banach.\\
\end{prop}
\begin{proof}
Sea $\bigl\{x_n\bigr\}_{n=1}^\infty:x_1,x_2,x_3,\ldots\ldots,$ una S. de Cauchy en $\mathfrak{L}^2$\\\\
i.e, $x_1=\left(\xi_1^1,\xi_2^1,\xi_3^1,\ldots\ldots,\xi_j^i,\ldots\right)$ S. de Cauchy en $\sum\limits_{j=1}^\infty|\xi_j^1|^2<\infty\\\\
x_2=\left(\xi_1^2,\xi_2^2,\xi_3^2,\ldots\ldots,\xi_j^2,\ldots\right)$ S. de Cauchy en $\sum\limits_{j=1}^\infty|\xi_j^2|^2<\infty\\\\
x_3=\left(\xi_1^3,\xi_2^3,\xi_3^3,\ldots\ldots,\xi_j^3,\ldots\right)$ S. de Cauchy en $\sum\limits_{j=1}^\infty|\xi_j^3|^2<\infty\\\\
\ldots\ldots$\\\\
Sea $\varepsilon>0.$ Como $\bigl\{x_n\bigr\}_{n=1}^\infty$ es una S. de Cauchy en $\mathfrak{L}^2,\exists N\in\mathbb{N}$ tal que $\forall m,n>N:{\left(\sum\limits_{j=1}^\infty|\xi_j^m-\eta_j^n|^2\right)}^{1/2}=\|x_m-x_n\|_2<\varepsilon$\\\\
i.e, $\forall m,n>N,\forall j=1,2,\ldots\ldots:|\xi_j^m-\eta_j^n|<\varepsilon^2$\\\\
o tambièn $\forall m,n>N,\forall j=1,2,\ldots\ldots:|\xi_j^m-\eta_j^n|<\varepsilon.$\\\\
Asì que si fijamos $j, j=1,2,3,\ldots\ldots$ la columna \begin{tabular}{c}
$\xi_j^1$\\
$\xi_j^2$\\
$\vdots$\\
$\xi_j^m$\\
$\xi_j^n$\\
$\downarrow$\\
$\xi_j$\\\\
\end{tabular}
es una S. de Cauchy en K y como K es Banach, las sucesiòn columna $j^a$ converge. Llamemos $\xi_j=\lim\limits_{m\longrightarrow\infty}\xi_j^m$\\\\
Esto permite que podamos definir la sucesiòn $x=\left(\xi_1,\xi_2,\ldots\ldots,\right)$ ò \\$x=\left(\lim\limits_{m\longrightarrow\infty}\xi_1^m,\lim\limits_{m\longrightarrow\infty}\xi_2^m,\ldots\ldots,\right)$\\\\
Vamos ahora a dm. que\\
\begin{enumerate}
\item $x=\bigl\{\xi_j\bigr\}_{j=1}^\infty\in\mathfrak{L}^2$\\
\item $\lim\limits_{m\longrightarrow\infty}=x$ con lo que quedarà establecido que $\mathfrak{L}^2$ es Banach.\\\\
\item[1] Fijemos $p\in\mathbb{N}.$\\
Sea $\varepsilon>0.$ Entonces $\dfrac{\varepsilon}{2\sqrt{p}}>0.$\\\\
Como $\xi_1^m\overset{m\rightarrow\infty}{\longrightarrow\xi_1},\exists N\in\mathbb{N}^1$ tal que $\forall m>N^1:|\xi_1^m-\xi_1|<\dfrac{\varepsilon}{2\sqrt{p}}\\\\
\xi_2^m\overset{m\rightarrow\infty}{\longrightarrow\xi_1},\exists N\in\mathbb{N}^2$ tal que $\forall m>N^2:|\xi_2^m-\xi_2|<\dfrac{\varepsilon}{2\sqrt{p}}\\\\
\vdots\\\\
\xi_p^m\overset{m\rightarrow\infty}{\longrightarrow\xi_p},\exists N\in\mathbb{N}^p$ tal que $\forall m>N^p:|\xi_p^m-\xi_p|<\dfrac{\varepsilon}{2\sqrt{p}}$\\\\
Asì que si escogemos $N_2>\max\bigl\{N^1,N^2,\ldots\ldots,N^p\bigr\}$ se tendrà que $\forall m>N_2:|\xi_j^m-\xi_j|<\dfrac{\varepsilon}{2\sqrt{p}},j=1,2,\ldots\ldots,p$ y al colacar y sumar sobre j de a p, $\left(\sum\limits_{j=1}^p|\xi_m^j-\xi_j|^2\right)<\dfrac{\varepsilon^2}{4}$ siempre que $m>N_2$ i.e, $\exists N_2\in\mathbb{N}$ tal que $\forall n>N_2:{\left(\sum\limits_{j=1}^p|\xi_m^j-\xi_j|^2\right)}^{1/2}<\dfrac{\varepsilon}{2}$\\\\
Por lo tanto, si $\underset{fijo}{m}>N_2\\\\
\left\|(\xi_1,\xi_2,\ldots,\xi_p)\right\|_2=\left\|(\xi_1,\xi_2,\ldots,\xi_p)-(\xi_1^m,\xi_2^m,\ldots,\xi_p^m)+(\xi_1^m,\xi_2^m,\ldots,\xi_p^m)\right\|_2\\\\
\leqslant\left\|(\xi_1^m-\xi_1,\ldots\ldots,\xi_p^m-\xi_p)\right\|_2+\left\|(\xi_1^m,\ldots\ldots,\xi_p^m)\right\|_2$\\\\
i.e, ${\left(\sum\limits_{j=1}^p|\xi_j|^2\right)}^{1/2}\leqslant{\left(\sum\limits_{j=1}^p|\xi_j^m-\xi_j|^2\right)}^{1/2}+{\left(\sum\limits_{j=1}^p|\xi_j^m|^2\right)}^{1/2}<\dfrac{\varepsilon}{2}+\|x_m\|_2\\\\
\Longrightarrow\sum\limits_{j=1}^2|\xi_j|^2\leqslant\left(\dfrac{\varepsilon}{2}+\|x_m\|_2\right)^2,\forall p\in\mathbb{N}.$\\\\
lo cual significa que es cota superior de la S. Sumas parciales de la serie $|\xi_1|^2+|\xi_2|^2+\ldots\ldots$ O sea que la serie $|\xi_1|^2+|\xi_2|^2+\ldots\ldots$ converge, i.e, $x=\bigl\{\xi_j\bigr\}_{j=1}^\infty\in\mathfrak{L}^2.$\\\\
\item [2.] Veamos finalmente que $\lim\limits_{m\rightarrow\infty}x_m=x,$ o que $\forall \epsilon>0\,\,\exists N\in\mathbb{N}$ tal que $\forall n>N:\|x_m-x\|_2<\epsilon,$ i.e, ${\left(\sum\limits_{j=1}^\infty|\xi_j^m-\xi_j|^2\right)}^{1/2}<\epsilon$ siempre que $m>N.$\\\\
Segùn (1), dado $\epsilon>0,\exists N\in\mathbb{N}$ tal que $\forall m>N, p\in\mathbb{N},{\left(\sum\limits_{j=1}^\infty|\xi_j^m-\xi_j|^2\right)}^{1/2}<\epsilon\Longrightarrow\left(\sum\limits_{j=1}^\infty|\xi_j^m-\xi_j|^2\right)<\epsilon^2,\forall p\in\mathbb{N}\Longrightarrow{\left(\sum\limits_{j=1}^\infty|\xi_j^m-\xi_j|^2\right)}^{1/2}<\epsilon$ siempre que $m>N.$\\
\end{enumerate}
\end{proof}
\begin{prop}
El dual de $\mathfrak{L}^p$ es $\mathfrak{L}^q,$ o sea que $\mathfrak{L}^{p'}:$ el dual de $\mathfrak{L}^p$ es isomètricamente isomorfo a $\mathfrak{L}^q.$\\
\end{prop}
\begin{proof}
Sea\\
$\begin{diagram}
\node[2]{\phi:\mathfrak{L}^q}\arrow{e,t}\\
\node{\mathfrak{L}^{p'}}
\end{diagram}$\\
$\begin{diagram}
\node[2]{f}\arrow{e,t}\\
\node{\phi(f):\mathfrak{L}^p\longrightarrow\mathbb{K}}
\end{diagram}$\\
$\begin{diagram}
\node[3]{x}\arrow{e,t}\\
\node{\phi_f(x)=\sum\limits_{n=1}^\infty\alpha_n\xi_n}
\end{diagram}$\\\\
Sea $f=\bigl\{\alpha_n\bigr\}_{n=1}^\infty\in\mathfrak{L}^q.$ Veamos que $\phi(f)$ està bien definida. En primer lugar, por la Desigualdad de H\"{o}lder:$\sum\limits_{j=1}^n|\alpha_j\xi_j|\leqslant{\left(\sum\limits_{j=1}^n|\alpha_j|^q\right)}^{1/q}{\left(\sum\limits_{j=1}^n|\xi_j|^p\right)}^{1/p}\leqslant\|f\|_q\|x\|_p$\\\\
Lo que nos dm. que la serie $\sum\limits_{n=1}^\infty\alpha_n\xi_n$ es ABS. Convergente, y por tanto la serie $\sum\limits_{n=1}^\infty\alpha_n\xi_n$ converge. Esto prueba que $\phi(f)$ està bien definida, tenièndose ademàs que \begin{gather}\label{11}\left|\sum\limits_{n=1}^\infty\alpha_n\xi_n\right|\leqslant\sum\limits_{n=1}^\infty|\alpha_n\xi_n|\leqslant\|f\|_q\|x\|_p\end{gather}\\\\
Sea $x=\bigl\{\xi_n\bigr\}_{n=1}^\infty,y=\bigl\{\beta_n\bigr\}_{n=1}^\infty\in\mathfrak{L}^p.$ Veamos que $\phi_f(x+y)=\phi_f(x)+\phi_f(y).$\\\\
$x+y=\bigl\{\xi_n+\beta_n\bigr\}_{n=1}^\infty\in\mathfrak{L}^p$ y por tanto $\phi_f(x+y)=\sum\limits_{n=1}^\infty\alpha_n(\xi_n+\beta_n)=\sum\limits_{n=1}^\infty\alpha_n\xi_n+\sum\limits_{n=1}^\infty\alpha_n\beta_n=\phi_f(x)+\phi_f(y).$\\\\
Esto dm. que $\phi_f\in\mathfrak{L}^{p*}:\text{dual algebraico de $\mathfrak{L}^p$}.$\\\\
Segùn ~\eqref{11} \begin{gather}\forall x\in\mathfrak{L}^p:\left|\phi_f(x)\right|=\left|\sum\limits_{n=1}^\infty\alpha_n\xi_n\right|\leqslant\sum\limits_{n=1}^\infty|\alpha_n\xi_n|\leqslant\|f\|_q\|x\|_p\end{gather}\\\\
lo que dm. que $\forall f\in\mathfrak{L}^q:\phi_f$ es continua, i.e, $\forall f\in\mathfrak{L}^q:\phi_f\in\mathfrak{L}^{p'}$\\\\
Asì que\\
$\begin{diagram}
\node[2]{\phi:\mathfrak{L}^q}\arrow{e,t}\\
\node{\mathfrak{L}^{p'}}
\end{diagram}$\\
$\begin{diagram}
\node[2]{f}\arrow{e,t}\\
\node{\phi_f}
\end{diagram}$\\\\
està bien definida.\\
Veamos ahora que $\phi$ es A.L.\\\\
Sean $f=\bigl\{\alpha_n\bigr\}_{n=1}^\infty,g=\bigl\{\theta_n\bigr\}_{n=1}^\infty\in\mathfrak{L}^q.$ Veamos que \begin{gather}\label{12}\phi_{(f+g)}=\phi_f+\phi_g\end{gather}\\\\
$f+g=\bigl\{\alpha_n+\theta_n\bigr\}_{n=1}^\infty\in\mathfrak{L}^q$ y por tanto $\phi_{f+g}\in\mathfrak{L}^{p'}.$\\\\
Tomemos $x=\bigl\{\xi_n\bigr\}_{n=1}^\infty\in\mathfrak{L}^p.$ Para dm, ~\eqref{12} bastarà con probar que $\phi_{f+g}(x)=\phi_f(x)+\phi_g(x)$\\\\
$\phi_{f+g}(x)=\sum\limits_{n=1}^\infty(\alpha_n+\theta_n)\xi_n=\sum\limits_{n=1}^\infty\alpha_n\xi_n+\sum\limits_{n=1}^\infty\theta_n\xi_n=\phi_f(x)+\phi_g(x).$\\ Esto dm. que $\phi\in\mathcal{L}\left(\mathfrak{L}^q,\mathfrak{L}^{p'}\right)$\\\\
Veamos ahora que $\phi$ es continua.\\
$\forall f\in\mathfrak{L}^q,\phi_f\mathfrak{L}^{p'},\\
\begin{diagram}
\node[2]{\phi_f:\mathfrak{L}^p}\arrow{e,t}\\
\node{\mathbb{K}}
\end{diagram}$\\
$\begin{diagram}
\node[2]{x}\arrow{e,t}\\
\node{\phi_f(x)}
\end{diagram}$\\
Debemos dm. que $\exists M>0$ tal que $\forall f\in\mathfrak{L}^q:\|\phi_f\|\leqslant M\|f\|_q$\\\\
Sea $f\in\mathfrak{L}^q.$ Entonces $\phi_f\in\mathfrak{L}^{p'}$ y \begin{gather}\label{13}\|\phi_f\|=\sup\limits_{\|x\|_p\leqslant1}\left|\phi_f(x)\right|=\sup\limits_{\|x\|\leqslant1}\|f\|_q\|x\|_q=\|f\|_q\end{gather}\\\\
$\left|\phi_f(x)\right|\leqslant\|f\|_q\|x\|_p.$\\\\
Asì que \begin{gather}\label{14}\forall f\in\mathfrak{L}^q:\|\phi_f\|\leqslant1\|f\|_q\end{gather}\\\\
lo que nos dm. que $\phi\in\mathcal{L}_c\left(\mathfrak{L}^q,\mathfrak{L}^{p'}\right).$\\\\
Veamos que $\phi$ es sobre y que \begin{gather}\label{15}\|\phi_f\|\geqslant\|f\|_q\end{gather}\\\\
Una vez establezcamos lo anterior, de ~\eqref{14} y ~\eqref{15} se concluye que $\|\phi_f\|=\|f\|$ y por tanto $\phi$ es isometrìa. Luego $\phi$ es A.L. biyectiva y continua y por tanto $\mathfrak{L}^q=\mathfrak{L}^{p'}$(isomorfismo isomètrico)\\\\
Sea $\overline{f}\in\mathfrak{L}^{p'}.$ Veamos que: $\exists f\in\mathfrak{L}^p$ tal que $\phi(f)=\overline{f}.$\\
$\begin{diagram}
\node[2]{\overline{f}:\mathfrak{L}^p}\arrow{e,t}\\
\node{\mathbb{K}}
\end{diagram}$\\
$\begin{diagram}
\node[2]{x}\arrow{e,t}\\
\node{\overline{f}(x)}
\end{diagram}$\\\\
es A.L. continua.\\\\
Como $\forall n\in\mathbb{N}:e_n(0,0,\ldots,1,0,\ldots)\in\mathfrak{L}^p,\overline{f}(e_n)\in\mathbb{K}.$\\\\
\begin{gather}\label{16}\overline{f}(e_n)=\alpha_n\in\mathbb{K}=|\alpha_n|e^{i\theta_n}\diagup0\leqslant\theta\leqslant{360}^\circ\end{gather}\\
los $|\alpha_n|$ y $\theta_n$ conocidos.\\\\
Fijemos $m\in\mathbb{N}$ y definamos $\beta_k=|\alpha_k|^{q-1}e^{-i\theta_k},k=1,\ldots,m\Longrightarrow|\beta_k|=|\alpha_k|^{q-1}$\\\\
Consideremos ahora la sucesiòn:\\
$\bigl\{w_n\bigr\}_{n=1}^\infty=(\beta_1,\beta_2,\ldots\ldots,\beta_m,0,0,\ldots\ldots)=\left(|\alpha_1|^{q-1}e^{-i\theta_1},|\alpha_2|^{q-1}e^{-i\theta_2},\ldots\ldots,|\alpha_m|^{q-1}e^{-i\theta_m},0,0,\ldots\ldots\right)\\\\
=\left(|\alpha_1|^{q-1}e^{-i\theta_1}\right)_{e_1}+\left(|\alpha_2|^{q-1}e^{-i\theta_2}\right)_{e_2}+\ldots\ldots+\left(|\alpha_m|^{q-1}e^{-i\theta_m}\right)_{e_m}$\\\\
Es claro que $\bigl\{w_n\bigr\}_{n=1}^\infty\in\mathfrak{L}^p$ ya que tiene una cola de ceros. \\Luego $\overline{f}_{(w_n)}=|\alpha_1|^{q-1}e^{-i\theta_1}\overline{f}_{e_1}+|\alpha_2|^{q-1}e^{-i\theta_2}\overline{f}_{e_2}+\ldots\ldots,|\alpha_m|^{q-1}e^{-i\theta_m}\overline{f}_{e_m}\\\\
~\eqref{16}=|\alpha_1|^{q-1}e^{-i\theta_1}|\alpha_1|e^{i\theta_1}+|\alpha_2|^{q-1}e^{-i\theta_2}|\alpha_2|e^{i\theta_2}+\ldots\ldots+|\alpha_m|^{q-1}e^{-i\theta_m}|\alpha_m|e^{i\theta_m}\\\\
=|\alpha_1|^q+\ldots\ldots+|\alpha_m|^q=\sum\limits_{k=1}^\infty|\alpha_k|^q$\\\\
Como $\overline{f}\in\mathfrak{L}^{p'},\text{y $\bigl\{w_n\bigr\}\in\mathfrak{L}^p$};\left|\overline{f}_{(w_n)}\right|\leqslant\left\|\overline{f}\right\|\left\|\bigl\{w_n\bigr\}\right\|_p$\\\\
O sea que $\sum\limits_{k=1}^m|\alpha_k|^q\leqslant\|\overline{f}\|{\left(\sum\limits_{k=1}^m|\beta_k|^p\right)}^{1/p}\\\\
=\|\overline{f}\|{\left(\sum\limits_{k=1}^m{\left(|\alpha_k|^{q-1}\right)}^p\right)}^{1/p}=\|\overline{f}\|{\left(\sum\limits_{k=1}^m|\alpha_k|^q\right)}^{1/q}\\\\\\
\frac{1}{p}+\frac{1}{q}=1\\\\
p+q=pq\\\\
(q-1)p=pq-p=p+q-p=q\hspace{0.5cm}\therefore\hspace{0.5cm}{\left(\sum\limits_{k=1}^m|\alpha_k|^q\right)}^{1-\frac{1}{p}}\leqslant\|\overline{f}\|$\\\\
i.e ${\left(\sum\limits_{k=1}^m|\alpha_k|^q\right)}^{\frac{1}{p}}\leqslant\|\overline{f}\|,$ cualqiuera sea $m\in\mathbb{N}$ lo cual significa que $f=\bigl\{\alpha_n\bigr\}_{n=1}^\infty\in\mathfrak{L}^q$ y que \begin{gather}\label{17}\|f\|_q\leqslant\|\overline{f}\|\end{gather}\\\\
Veamos ahora que \begin{gather}\label{18}\phi(f)=\overline{f}\end{gather}\\\\
Como $f=\bigl\{\alpha_n\bigr\}_{n=1}^\infty\in\mathfrak{L}^q,{\left(\sum\limits_{n=1}^\infty|\alpha_n|^q\right)}^{1/q}=\|f\|_q$\\\\
$\begin{diagram}
\node[2]{\mathfrak{L}^{p'}\ni\phi_f:\mathfrak{L}^p}\arrow{e,t}\\
\node{\mathbb{K}}
\end{diagram}$\\
$\begin{diagram}
\node[2]{x}\arrow{e,t}\\
\node{\phi_f(x)=\sum\limits_{n=1}^\infty\alpha_n\xi_n}
\end{diagram}$\\\\
$\begin{diagram}
\node[2]{\mathfrak{L}^{p'}\in\overline{f}:\mathfrak{L}^p}\arrow{e,t}\\
\node{\mathbb{K}}
\end{diagram}$\\
$\begin{diagram}
\node[2]{e_n}\arrow{e,t}\\
\node{\overline{f}_{(e_n)}=\alpha_n}
\end{diagram}$\\\\
Para obtener ~\eqref{18} debemos dm. que $\forall x=\bigl\{\xi_n\bigr\}_{n=1}^\infty\in\mathfrak{L}^p:\phi_f(x)=\overline{f}(x).$\\\\
Tomemos $x=\bigl\{\xi_n\bigr\}_{n=1}^\infty\in\mathfrak{L}^p$ Entonces $\phi_f(x)=\sum\limits_{n=1}^\infty\alpha_n\xi_n$\\\\
Como $x=\bigl\{\xi_n\bigr\}_{n=1}^\infty\in\mathfrak{L}^p,S_m=\sum\limits_{n=1}^m|\xi_n|^p\overset{m\rightarrow\infty}{\longrightarrow}\sum\limits_{n=1}^\infty|\xi_n|^p$\\\\
Sea $\epsilon>0.$ Entonces $\exists M\in\mathbb{N}$ tal que $\forall m>M:\left|S_m-\sum\limits_{n=1}^\infty|\xi_n|^p\right|<{\epsilon}^p,$ i.e, $\sum\limits_{n=m+1}^\infty|\xi_n|^p<{\epsilon}^p$ siempre que $m>M.$ O sea que \begin{gather}\label{19}\|x-x_m\|_p={\left(\sum\limits_{n=m+1}^\infty|\xi_n|^p\right)}^{1/p}<\epsilon\end{gather}\\\\
siempre que $m>M.$ $x_m=(\xi_1,\xi_2,\ldots\ldots,\xi_m,0,0,\ldots\ldots,)\in\mathfrak{L}^p;(x-x_m)\in\mathfrak{L}^p$\\\\
De otra parte,$x_m=(\xi_1,\xi_2,\ldots\ldots,\xi_m,0,0,\ldots\ldots,)\in\mathfrak{L}^p=\xi_1e_1+\xi_2e_2+\ldots\ldots+\xi_me_m\\\\
\therefore\hspace{0.5cm}\overline{f}(x_m)=\xi_1\overline{f}(e_1)+\xi_2\overline{f}(e_2)+\ldots\ldots+\xi_m\overline{f}(e_m)=\alpha_1\xi_1+\alpha_2\xi_2+\ldots\ldots+\alpha_m\xi_m=\sum\limits_{n=1}^m\alpha_n\xi_n\\\\
\|\overline{f}(x)-\overline{f}{x_m}\|=\|\overline{f}{x-x_m}\|\leqslant\|\overline{f}\|\|x-x_m\|_p\underset{~\eqref{19}}{\leqslant}\|f\|\epsilon,\\\\
\left|\sum\limits_{n=1}^\infty\alpha_n\xi_n-\sum\limits_{n=1}^m\alpha_n\xi_n\right|=\left|\sum\limits_{n=m+1}^\infty\right|\leqslant\sum\limits_{n=m+1}^\infty|\alpha_n\xi_n|\leqslant\|f\|_q\|x-x_m\|_p<\|\overline{f}\|\epsilon$ \\\\siempre que $m>M$\\\\
$\left|\overline{f}(x)-\sum\limits_{n=1}^\infty\alpha_n\xi_n\right|\leqslant\left|\overline{f}(x)-\sum\limits_{n=1}^m\alpha_n\xi_n\right|+\left|\sum\limits_{n=1}^m\alpha_n\xi_n-\sum\limits_{n=1}^\infty\alpha_n\xi_n\right|\\\\
\|\overline{f}\|\epsilon+\epsilon\|\overline{f}\|=2\|\overline{f}\|\epsilon,\hspace{0.5cm}\forall \epsilon$\\\\
Luego, $\overline{f}(x)=\sum\limits_{n=1}^\infty\alpha_n\xi_n.$\\
\end{proof}
\section{Mapeos Bilineales}
En esta secciòn se tratan los Mapeos Bilineales y se realiza un especial ènfasis al respecto sobre las diferencias entre èstos y los Mapeos Lineales; aunque estos estàn relacionados ìntimamente con el tema.\\\\
Sean $E,F,G$ esp.vectoriales sobre algùn campo escalar $\mathbb{K}=\mathbb{R}$ ò $\mathbb{C}$ de nùmeros reales o complejos. Un mapeo\\
$\begin{diagram}
\node[2]{\Phi:E\times F}\arrow{e,t}\\
\node{G}
\end{diagram}$\\\\
es denominado $\textit{bilineal}$ si los mapeos\\
$\begin{diagram}
\node[2]{\phi_x:F}\arrow{e,t}\\
\node{G}
\end{diagram}\hspace{0.5cm}\text{y}\hspace{0.5cm}\begin{diagram}
\node[2]{\phi_y:E}\arrow{e,t}\\
\node{G}\\
\node[2]{x}\arrow{e,t}\\
\node{\phi(x,y)}
\end{diagram}$\\
$\begin{diagram}
\node[2]{y}\arrow{e,t}\\
\node{\phi(x,y)}
\end{diagram}$\\\\
son lineales $\forall x\in E,F,$ sìmbolos: $\phi\in\mathsf{Bil}(E,F,G)$ si $$\phi_x\in L(F,G)\,\,\text{y $\phi_y\in L(E,G)$}$$\\\\
$\forall x\in E\,\,\text{y}\,\, y\in F.$ Por simplicidad: $\mathsf{Bil}(E,F)=\mathsf{Bil}(E,F,\mathbb{K}).$ Si $E,F,G$ son espacios normados (màs generalmente espacios vectoriales topològicos), el conjunto de mapeos bilineales continuos $E\times F\longrightarrow G$ se puede denotar por $\mathsf{Bil}(E,F,G)$ y $\mathsf{Bil}(E,F)$ si $G=\mathbb{K}.$\\\\
Para $$\phi(x,y)-\phi(x_0,y_0)=\phi(x-x_0,y-y_0)+\phi(x_0,y-y_0)$$\\\\
Los siguientes desarrollos son sencillos de deducir.\\\\
\begin{prop}
Para $\phi\in\mathsf{Bil}(E,F;G)$ las siguientes afirmaciones son equivalentes entre sì:\\
\begin{enumerate}
\item[(a)] $\phi$ es continua\\
\item[(b)] $\phi$ es continua (0,0)\\
\item[(c)] Sea C una constante, tal que $C\geqslant0,$ se tiene $\|\phi(x,y)\|_G\leqslant C\|x\|_E\|y\|_F,\forall (x,y)\in E\times F$\\\\
\end{enumerate}
\end{prop}
Se observa claramente que $\|\phi\|=\min\bigl\{C\geqslant0\bigr\}=\sup\bigl\{\|\phi(x,y)\|_G\diagup x\in B_E, y\in B_F\bigr\}$\\
define una norma sobre $\mathsf{Bil}(E,F;G)$ que es uniforme y una norma completa si $\|.\|_G$ lo es.\\\\
Nòtese que el mapeo Bilineal continuo no es $\textit{uniformemente continuo}$ puesto que, la resticciòn $\mathbb{R}^2\longrightarrow\mathbb{R},(x,y)\rightsquigarrow xy$ sobre la diagonal es la funciòn $\mathbb{R}\ni x\rightsquigarrow x^2\in\mathbb{R}.$\\\\
Un mapeo Bilineal $\phi\in\mathsf{Bil}(E,F;G)$ es $\textit{separable continuo}$ si para todo $\phi_x:F\longrightarrow G$ y $\phi_y:E\longrightarrow G$ son continuos.\\\\
\begin{teor}
Sean $E,F,G$ espacios normados y E es completo. Para todo mapeo Bilineal continuo separable $\phi\in\mathsf{Bil}(E,F;G)$ es continuo.\\
\end{teor}
\begin{proof}
El conjunto $D=\bigl\{z'\circ\phi_y\diagup z'\in B_G',y\in B_F\bigr\}\subset E'$ es $\sigma(E',E)-$ continuo puesto que para todo $x\in E$ $$\left|\langle z'\circ_y,x\rangle\right|=\left|\langle z',\phi(x,y)\rangle\right|\leqslant\|z'\|\|\phi(x,y)\|\leqslant\|x\phi\|$$\\\\
Por teorema de MacKey's $\diagup$ la uniformidad principal continua muestra que D es $\textit{uniformemente continuo},$ i.e allì la constante $c\geqslant0$ tal que $\forall z'\in B_G'$ y $y\in B_F$ $$\left|\langle z',\phi(x,y)\rangle\right|=\left|\langle z'\circ\phi_y,x\rangle\right|\leqslant c\|x\|_E$$\\\\
para todo $x\in E.$ Esto prueba que $\|\phi\|\leqslant c.$\\\\
\end{proof}
Algunos ejemplos de mapeos Bilineales:\\
\begin{enumerate}
\item Para $x'\in E'$ y $y'\in F'$ $$[x'\underline{\otimes} y'](x,y)=\langle x',x\rangle\langle y',y\rangle$$\\\\
define una forma Bilineal continua y $\|x'\underline{\otimes} y'\|=\|x'\|\|y'\|.$ Si $x'_n$ y $y'_n$ pertenecen a una bola unitaria y $(\lambda_n)\in\mathcal{L}_1$, entonces $\varphi(x,y)=\sum\limits_{n=1}^\infty\lambda_n [x'\underline{\otimes} y'](x,y)$ es una funciòn bien definida y $\|\varphi\|\leqslant\sum\limits_{n=1}^\infty|\lambda_n|;$ pertenecen a la clase de las formas bilineales denominadas $\textit{nucleares}.$\\\\
\item La condiciòn del mapeo sobre el espacio $\mathcal{L}(E,F)$ de operadores lineales continuos\\
$\begin{diagram}
\node{\mathcal{L}(E,F)\times E}\arrow{e,t}\\
\node{F}\\
\node{T,x}\arrow{e,t}\\
\node{Tx}
\end{diagram}$\\\\
tiene norma 1 (si E y F son no triviales).\\\\
\item Si E y F son esp. vectoriales de dimensiòn finita, entonces todo mapeo bilineal $E\times F\longrightarrow G$ es continuo (empleando bases).\\\\
\item El mapeo de convoluciòn\\
$\begin{diagram}
\node{L_1(\mathbb{R})\times L_1(\mathbb{R})}\arrow{e,t}\\
\node{L_1(\mathbb{R})}\\
\node{(f,g)}\arrow{e,t}\\
\node{f*g}
\end{diagram}$\\
es bilineal.\\\\
\item Tomemos las funciones continuas sobre un espacio compacto K y sea E un espacio normado. Entonces\\
$\begin{diagram}
\node{C(K)\times E}\arrow{e,t}\\
\node{C(K,E)}\\
\node{f,x}\arrow{e,t}\\
\node{f(.)x}
\end{diagram}$\\
es bilineal.\\\\
Los mapeos\\
$\begin{diagram}
\node{\mathsf{Bil}(E,F)}\arrow{e,t}\\
\node{L(E,F^{*})}\\
\node{\varphi}\arrow{e,t}\\
\node{L\varphi}\\
\node{\langle L\varphi x,y\rangle=\varphi(x,y)}
\end{diagram}$\\\\
$\begin{diagram}
\node{L(E,F^{*})}\\
\node{T}\arrow{e,t}\\
\node{\beta T}\\
\node{\beta T(x,y)=\langle Tx,y\rangle}
\end{diagram}$\\\\
son espacios vectoriales isomorfos y son inversos a cualquier otro. Puesto que $\|\varphi\|=\sup\bigl\{\left|\varphi(x,y)\right|\diagup x\in B_E, y\in B_F\bigr\}=\sup\bigl\{\|L\varphi x\|\diagup x\in B_F\bigr\}=\|L\varphi\|\in[0,\infty];$ este isomorfismo reduce las formas bilineales continuas a espacios isomètricos normados $$\mathsf{Bil}(E,F)=\mathcal{L}(E,F')$$ $$\|l\varphi\|=\|\varphi\|\,\,\text{y $\|\beta T\|=\|T\|$}$$\\\\
Esta relaciòn es bàsica para la compresiòn de las ideas que desarrollaremos acontinuaciòn: La forma bilineal continua sobre $E\times F$ son exàctamente los operadores lineales continuos $E\longrightarrow F'.$\\\\
El Teorema de Hahn-Banach para operadores, no es completo para formas bilineales continuas en el siguiente sentido: Sea $G\subset E$ un subespacio y $\varphi\in\mathsf{Bil}(G,F)$;\\?`No existe allì una extensiòn $\overline{\varphi}\in\mathsf{L}(E,F')$ de $\varphi$? Esto puede pensarse, por la identificaciòn de las formas bilineales y los operadores; $\forall T\in\mathcal{L}(G,F')$ puede tener una extensiòn $\overline{T}\in\mathcal{L}(E,F').$\\
\end{enumerate}
Asì se pueden observar algunos ejemplos de operadores que no son extensiòn del caso especial dado en que $G=F'$ y $T=idG$. La identidad del mapeo: La extensiòn $\overline{T}$ puede ser una proyecciòn de E sobre G.\\\\
\begin{itemize}
\item Debido a un resultado famoso de Lindenstrauss-Tzafriri [Sobre el problema del complemento de los subespacios; Israel J. Math. 9(1971) 263-69] todos los espacios de Banach de dimensiòn infinita que no son isomorfos al espacio de Hilbert no son complemento de subespacios cerrados.\\
\item Puede observarse màs en concreto en el ejemplo dado por la funciòn de Rademacher definida sobre $[0,1]$ $$r_n(t)=(-1)^k\text{sì $t\in\left[\dfrac{k}{2^n},\dfrac{k+1}{2^n}\right[$}$$\\\\
(La forma es ortonormal al sistema en $L_2[0,1],$ medida de Lebesgue) y considèrese la inyecciòn\\
$\begin{diagram}
\node[2]{\mathcal{L}_2}\arrow{e,t}\\
\node{L_1[0,1]}\\
\node[2]{(\xi_n)}\arrow{e,t}\\
\node{\sum\limits_{n=1}^\infty\xi_n\gamma_n}
\end{diagram}$\\\\
La desigualdad de $\textit{Khintchine}:$ ''Para $1\leqslant p<\infty$ allì son constantes $a_p$ y $b_p\geqslant1$ tal que $a_p^{-1}{\left(\sum\limits_{k=1}^n|\alpha_k|^2\right)}^{1/2}\leqslant{\left(\int\limits_{D_n}{|\sum\limits_{k=1}^n\alpha_k\xi_k(w)|}^p\mu_n(dw)\right)}^{1/2}\leqslant b_p{\left(\sum\limits_{k=1}^n|\alpha_k|^2\right)}^{1/2}$\\\\
para todo $n\in\mathbb{N}$ y $\alpha_1,\ldots,\alpha_n\in\mathbb{C}.$''\\\\
puede mostrarnos que $L_1$ induce una norma equivalente sobre $\mathcal{L}_2$\\\\
\end{itemize}
La extensiòn de las complexiones, afortunadamente, no es un problema. Òbservese que allì no son uniformemente continuas!\\
\begin{prop}
Sean E,F,G espacios normados y G completo. Para todo $\phi\in\mathsf{Bil}(E,F;G)$ existe una extensiòn ùnica $\overline{\phi}\in\mathsf{Bil}(\overline{E},\overline{F};G).$ Ademàs, $\|\phi\|=\|\overline{\phi}\|.$\\
\end{prop}
Este desarrollo es sencillo para la relaciòn isomètrica $$\mathsf{Bil}(E,F;G)=\mathcal{L}\left(E,\mathcal{L}(F,G)\right)$$\\
y la extensiòn de los operadores lineales continuos.\\\\
\begin{obser}
Si E es un espacio normado de dimensiòn menor que 2, entonces la multiplicaciòn por el escalar $k,\mathbb{K}\times E\longrightarrow E$ es bilineal, continua, sobreyectiva sòlo sì E es cerrado.\\
\end{obser}
\begin{proof}
Tomemos un conjuto abierto no vacìo $V\subset E$ y un funcional $y'\in E'$ con $$\inf|\langle y',x\rangle|>0$$\\
Si U es la bola unitaria abierta en $\mathbb{K},$ entonces $0\in U.V,$ sòlo si 0 no es un punto interior de $U.V$ puesto que $U.V\cap\ker y'=\bigl\{0\bigr\}$\\
\end{proof}
Esto tambièn es posible para los ejemplos $\phi\in\mathsf{Bil}(E,F;G)$ que son sobreyectivos y conjuntos cerrados en cero, i.e, cero pertenece al interior de $\phi(B_E,B_F).$\\\\
Otro propiedad negativa de los mapeos bilineales continuos es que ellos no permanecen continuos para las topologìas dèviles: los vectores unitarios $e_n$ en $\mathcal{L}_2$\,\,(en $\mathbb{R}$) convergen dèvilmente a cero exceptuando $(e_n| e_n)_{\mathcal{L}_2}=1$\\\\
Para espacios normados E y F se tiene el siguiente desarrollo isomètrico: $$\phi:\mathsf{Bil}(E,F)=\mathcal{L}(E,F')\hookrightarrow\mathcal{L}(F'',E')=\mathsf{Bil}(F'',E)=\mathsf{Bil}(E,F''),$$\\
$$T\rightsquigarrow T'$$\\
donde la ùltima igualdad es la ''transposiciòn'' obvia $U^t(x,y)=U(y,x).$ Para todo $\phi\in\mathsf{Bil}(E,F)$ se define ${\varphi}^\wedge=\phi(\varphi)\in\mathsf{Bil}(E,F'');$ si satisface $\|{\varphi}^\wedge\|=\|\varphi\|=\|L_{\varphi}\|$ y $${\varphi}^\wedge(x,y'')=\left\langle L'_\varphi(y''),x\right\rangle_{E',E}=\left\langle y'',L_\varphi(x)\right\rangle_{F'',F}=\left\langle y'',\varphi(x,.)\right\rangle_{F'',F}$$\\\\
para todo $(x,y'')\in E\times F''.$ Puesto que, por definiciòn $\varphi(x,y)=\langle y, L_\varphi(x)\rangle_{F,F''}$ para todo $x\in E$ y $y\in F,$ el mapeo ${\varphi}^\wedge$ extensiòn de $\varphi$ para $E\times F$ a $E\times F''$ con igual normal. ${\varphi}^\wedge$ es denominada $\textit{la extensiòn canònica derecha}$ de $\varphi$.\\\\
\begin{prop}
Sean E y F espacios normados y $\varphi\in\mathsf{Bil}(E,F).$ Entonces ${\varphi}^\wedge$ es la ùnica forma bilineal separada $\sigma(E,E')-\sigma(F'',F')$-mapeo continua $\psi:E\times F''\longrightarrow\mathbb{K}$ que extiende a $\varphi.$\\
\end{prop}
\begin{proof}
Que ${\varphi}^\wedge$ es una extensiòn es claro por el desarrolllo de la definiciòn para la ecuaciòn; se obtiene de las desigualdades para los funcionales $\sigma(F'',F')$- densidad de F en F''.\\
Claro, allì se tiene la extensiòn canònica izquierda $^\wedge{\varphi}$ sobre $E''\times F$ definida por $^\wedge{\varphi}={\left({(\varphi^t)}^\wedge\right)}^t$ dado por $$^\wedge{\varphi}(x'',y)=\langle x'',(L_\varphi\circ k_f)y\rangle_{E'',E'}=\langle x'',\varphi(.,y)\rangle_{E'',E'}$$\\
\end{proof}
De que manera ?`Son los funcionales ${(^\wedge{\varphi}})^\wedge$ y $^\wedge{({\varphi}}^\wedge)$ sobre $E''\times F''$ relativos? Bastante sorprendente, el siguiente desarrollo exacto:\\
\begin{corol}
Para $\varphi\in\mathsf{Bil}(E,F)$ el desarrollo de los tres estamentos siguientes son equivalentes:\\
\begin{enumerate}
\item Las dos extensiones ''canònicas'' ${(^\wedge{\varphi}})^\wedge$ y $^\wedge{({\varphi}}^\wedge)$ de $\varphi$ en $E''\times F''$ coinciden.\\
\item Allì $\psi\in\mathsf{Bil}(E'',F'')$ que es separable $\sigma(E'',E')-\sigma(F'',F')-$ son continuos y extensiones de $\varphi$.\\
\item $L_\varphi:E\longrightarrow F$ es compacto-dèvil.\\\\
En este caso el funcional $\psi$ en [2] aes igual a ${(^\wedge{\varphi}})^\wedge=^\wedge{({\varphi}}^\wedge)$\\
\end{enumerate}
\end{corol}
\begin{proof}
La proposiciòn implica sencillamente que $(a)\Leftrightarrow(b)$. Se observa la equivalencia de (a) y (c), por $$L{(^\wedge{\varphi}})^\wedge=K_{F'}\circ P_{F'}\circ L''_\varphi:E''\longrightarrow F''$$\\
$$L^\wedge{({\varphi}}^\wedge)=P_{F'''}\circ L''_{{\varphi}^\wedge}=P_{F'''}(K_F'\circ L_\varphi)''=$$\\
$$P_{F'''}\circ K''_{F'}\circ L''_\varphi=L''_\varphi$$
Ahora afirmamos que $L_\varphi$ es compacto-dèvil sii $L''_\varphi(E'')\subset F'.$\\
\end{proof}
\section{La teorìa algebraìca del producto tensorial.}
El objeto de estudio de los mapeos bilineales puede reducirse al estudio de los mapeos lineales. La construcciòn de estos nuevos espacios vectoriales $E\otimes F$ es, dada por el anàlisis, de una manera simple.\\\\
Para un conjunto arbitrario A definamos $\mathcal{F}(A)$ como el conjunto de todos las funciones $f:A\longrightarrow\mathbb{K}$ dentro de un soporte finito, i.e. $f(\alpha)=0$ excepto sobre un subconjunto finito de A. Para $\alpha\in A$ el $''\alpha-\text{vector unitario}''$ es la funciòn $e_\alpha\in\mathcal{F}(A)$ definida por $$e_{\alpha}(\beta)=\delta_{\alpha\beta}$$\\
del Delta de Kronecker. Es claro que para cada $f\in\mathcal{F}(A)$ se tiene la ùnica representaciòn $$f=\sum\limits_\alpha\in Af(\alpha)e_\alpha,$$\\
en otras palabras:$(e_\alpha)_{\alpha\in A}$ es un conjunto algebraico bàsico de $\mathcal{F}(A).$ Ahora tomemos dos conjunto A,B y consideremos, el mapeo bilineal\\
$\begin{diagram}
\node{\Psi_0:\mathcal{F}A\times\mathcal{F}B}\arrow{e,t}\\
\node{\mathcal{F}(A\times B)}\\
\node{f,g}\arrow{e,t}\\
\node{f(.),g(..)}
\end{diagram}$\\\\
Entonces $e(\alpha,\beta)=\Psi_0(e_\alpha,e_\beta),$ es claro que $$\text{ext im$\Psi_0$}=\mathcal{F}(A\times B)$$\\
Para $\psi\in\mathsf{Bil}(\mathcal{F}(A),\mathcal{F}(B),G)$ definamos un funcional $T\in L(\mathcal{F}(A\times B),G)$ por $$T(e(\alpha,\beta))=\psi(e_\alpha,e_\beta)$$\\
(una extensiòn lineal). Es obvio que T es el ùnico mapeo lineal $\mathcal{F}(A\times B)\longrightarrow G$ en el interior de $$T\circ\Psi_0=\psi$$\\
Esto demuestra que $$L(\mathcal{F}(A\times B),G)=\mathsf{Bil}(\mathcal{F}(A),\mathcal{F}(B);G)$$\\
$$T\rightsquigarrow T\circ\Psi_0$$\\
es un isomorfismo lineal de espacios vectoriales.

\end{document}